\documentclass [11pt,a4paper]{article}
\usepackage[ansinew]{inputenc}
\usepackage{amssymb}
\usepackage{amsmath}
\usepackage{graphicx}
\usepackage{amsthm}
\usepackage{cite}
\newtheorem{theorem}{\textbf{Theorem}}[section]
\newtheorem{lemma}{\textbf{Lemma}}[section]
\newtheorem{proposition}{\textbf{Proposition}}[section]
\newtheorem{corollary}{\textbf{Corollary}}[section]
\newtheorem{remark}{\textbf{Remark}}[section]
\newtheorem{definition}{\textbf{Definition}}[section]

\allowdisplaybreaks[4]

\def\be{\begin{equation}}
\def\ee{\end{equation}}
\def\bea{\begin{eqnarray}}
\def\eea{\end{eqnarray}}
\def\bt{\begin{theorem}}
\def\et{\end{theorem}}
\def\bl{\begin{lemma}}
\def\el{\end{lemma}}
\def\br{\begin{remark}}
\def\er{\end{remark}}
\def\bp{\begin{proposition}}
\def\ep{\end{proposition}}
\def\bc{\begin{corollary}}
\def\ec{\end{corollary}}
\def\bd{\begin{definition}}
\def\ed{\end{definition}}
 \def\non{\nonumber }

\begin{document}

\title{Strong solutions, global regularity and stability of a hydrodynamic system modeling vesicle and fluid interactions}
\author{
{\sc Hao Wu} \footnote{School of Mathematical Sciences and Shanghai
Key Laboratory for Contemporary Applied Mathematics, Fudan
University, 200433 Shanghai, China, Email:
\textit{haowufd@yahoo.com}.}\ \  and {\sc Xiang Xu}
\footnote{Department of Mathematical Sciences, Carnegie Mellon
University, Pittsburgh, PA 15213, Email:
\textit{xuxiang@math.cmu.edu}.}}
\date{\today}
\maketitle

\begin{abstract}
In this paper, we study a hydrodynamic system modeling the
deformation of vesicle membranes in incompressible viscous fluids. The system consists of the 
Navier--Stokes equations coupled with a fourth order phase-field equation. 
In the three dimensional case, we prove the existence/uniqueness of
local strong solutions for arbitrary initial data as well as global
strong solutions under the large viscosity assumption. We also
establish some regularity criteria in terms of the velocity for
local smooth solutions. Finally, we investigate the stability of the
system near local minimizers of the elastic bending energy.

{\bf Keywords.} phase-field equation, Naver--Stokes equations, well-posedness, regularity
criteria, stability.

\textbf{AMS Subject Classification}: 35B35, 35K30, 35Q35, 76D05.

\end{abstract}

\section{Introduction}\setcounter{equation}{0}
Biological vesicle membranes are interesting subjects widely studied
in biology, biophysics and bioengineering. They are not only
essential to the function of cells  but exhibit rich physical and
rheological properties as well \cite{MO05}. The single component
vesicles are possibly the simplest models for the biological cells
and molecules, which are formed by certain amphi-philic molecules
assembled in water to build bi-layers \cite{DLW04}.
   The equilibrium configurations of vesicle membranes can be characterized by the Helfrich bending elasticity energy of the surface \cite{Boal, Ci00,HE} such that they are minimizers of the bending energy under possible constraints like prescribed surface
area and bulk volume that account for the effects of density change
and osmotic pressure \cite{DLW04,Wang08}. Let $\Gamma$
 be a smooth, compact surface without boundary that represent the membrane of the vesicle.
  In the isotropic case, if the evolution
of the vesicle membrane does not change its topology, the
interfacial energy takes the following simplified form \cite{Ci00}:
 \be
E_{\mbox{elastic}}=\int_{\Gamma}\frac{k}{2}(H-H_0)^2\,ds, \non
  \ee
where $H$ is the mean curvature of the membrane surface; $k$ is
called the bending rigidity, which can depend on the local
heterogeneous concentration of the species; $H_0$ is the spontaneous
curvature that describes certain physical/chemical difference
between the inside and the outside of the membrane.

Recently, phase-field models have been derived within a general
energetic variational framework to study vesicle deformations and
numerical simulations of the membrane deformations were carried out
(see e.g., \cite{DLW04,DLRW09,DLW05,DLRW05a,DLW06,DW07,WD08,Wang08}
and references cited therein). As in \cite{DLRW05, DLW04}, we denote
by $\phi$ the phase function defined on the physical domain
$\Omega$, which is used to label the inside and the outside of the
vesicle $\Gamma$ such that $\phi$ takes the value $1$ inside of the
vesicle membrane and $-1$ outside. The sharp transition layer of the
phase function gives a diffusive interface description of the
vesicle membrane $\Gamma$, which is recovered by the level set $\{x:
\phi(x)=0\}$. The phase field approach describes geometric
deformations in Eulerian coordinates and it provides a convenient
way to capture topological transitions such as vessel fission or
fusion via changes in the level set topology.  This simplifies
numerical approximations because it suffices to consider a fixed
computational grid rather than tracking the position of the
interface \cite{DLRW09}.

For the sake of simplicity, we assume that $k$ is a positive
constant and $H_0=0$. The phase-field approximation of the Helfrich
bending elasticity energy is given by a modified Willmore energy
\cite{DLW04,Wang08} (see, e.g., \cite{DLRW05a} the approximation
energy for the elastic bending energy with non-zero spontaneous
curvature)
 \be\label{elastic bending energy}
 E_\varepsilon(\phi)=\dfrac{k}{2\varepsilon}\int_{\Omega}
 |f(\phi)|^2dx,\ \ \text{with} \ f(\phi)= -\varepsilon \Delta \phi+ \frac{1}{\varepsilon}(\phi^2-1)\phi,
 \ee
 where $\varepsilon$ is a small positive parameter (compared to the vesicle size) that characterizes the transition layer of the phase function.
 The convergence of the phase-field
model to the original sharp interface model as the transition width
of the diffuse interface $\varepsilon \rightarrow 0$ were carried
out in \cite{DLRW05,Wang08}. Two constraints are widely used in the
biophysical studies of vesicles \cite{Sei} such that the total
surface area and the volume of the vesicle are conserved (in time).
The former is a consequence of the incompressibility of the
membrane, while the latter is based on the consideration that, for a
fluctuating vesicle with the inside pressure and outside pressure
balanced by the osmotic pressure, the change in volume is normally a
much slower process in comparison with the shape change
\cite{DLRW09}. These two constraint functionals for the vesicle
volume and surface area are given by (cf. \cite{DLW04})
 \be\label{modified energy terms} A(\phi)=\int_{\Omega}\phi dx, \ \ \
B(\phi)=\int_{\Omega}\frac{\varepsilon}{2}|\nabla\phi|^2+\frac{1}{4\varepsilon}(\phi^2-1)^2dx,
\ee respectively. Two penalty terms are introduced in order to
enforce these constraints, and the approximate elastic bending
energy is formulated in the following form \cite{DLRW05,
DLRW05a,DLW05, DLW06}:
 \be\label{modified energy}
 E(\phi)=E_\varepsilon(\phi)+\frac12M_1(A(\phi)-\alpha)^2+\frac12M_2(B(\phi)-\beta)^2,
 \ee
where $M_1$ and $M_2$ are two penalty constants, $\alpha=A(\phi_0)$
and $\beta=B(\phi_0)$ are determined by the initial value of the
phase function $\phi_0$. Alternatively, Lagrange multipliers should be
introduced to preserve the vesicle volume and total vesicle surface area
\cite{DLRW09,DLW06}.

In this paper, we consider a hydrodynamic system for
 the interaction of a vesicle with the fluid field, which describes the evolution of vesicles immersed in an incompressible, Newtonian fluid \cite{DLL07}.
 More precisely, we study the following phase-field Navier--Stokes equations for the velocity field $u$ of the fluid and the phase function $\phi$:
\bea
  &&u_t+u\cdot\nabla u+\nabla P=\mu\Delta u+\dfrac{\delta
  E(\phi)}{\delta\phi}\nabla\phi,
  \label{NS}\\
  &&\nabla\cdot u=0, \label{incompressibility}\\
  &&\phi_t+u\cdot\nabla\phi=-\gamma\dfrac{\delta
  E(\phi)}{\delta\phi}. \label{phasefield}
\eea

System \eqref{NS}--\eqref{phasefield} can be obtained via an
energetic variation approach \cite{HKL,YFLS} (see \cite{DLRW09} for the
derivation of a corresponding evolution system that adopts the
Lagrange multiplier approach for the volume and surface area
constraints). The resulting membrane configuration and the flow
field reflect the competition and the coupling of the kinetic energy
and membrane elastic energies. Equations \eqref{NS} and
\eqref{incompressibility} are the Navier--Stokes equations of the
viscous incompressible fluid with unit density and a force term, which is
derived from the variation of the elastic bending energy and it
involves a nonlinear combination of higher-order spatial derivatives
of the phase function. The scalar function $P$ denotes the pressure.  We denote $\mu$ the
fluid viscosity, which is assumed to be a positive constant
throughout both fluid phases and the interface. The third equation \eqref{phasefield}
is a relaxed transport equation of $\phi$ under the velocity field
$u$. Its right-hand side contains a regularization term, where
$\gamma$ is the mobility coefficient that is assumed to be a small
positive constant.

Well-posedness of the system \eqref{NS}--\eqref{phasefield} subject
to no-slip boundary condition for the velocity field and Dirichlet
boundary conditions for the phase function has been studied in
\cite{DLL07, LYN}.  In \cite{DLL07}, the authors obtained the
existence of weak solutions by using the Galerkin method. They also
obtain the uniqueness of solutions in a more regular class than the
one used for existence. Quite recently, existence and uniqueness of
local strong solutions has been proved in \cite{LYN} by using the Banach fixed point theorem. Their result was obtained in a proper Sobolev space with fractional order for the phase function, since some compatibility conditions (on the boundary) are
required in the fixed point strategy described in \cite{LYN} in order to incorporate the estimates on the nonlinear force term.
Besides, \textit{almost} global existence of strong solutions were
obtained under the assumption that the initial data and the quantity
$(|\Omega| + \alpha)^2$ are sufficiently small (controlled by a small parameter $q=q(T)$ depending on the existence time $T$). However, they were
not able to prove global existence of strong solutions because uniform-in-time \textit{a
priori} estimates were not available in their argument. 

We note that, although Dirichlet type boundary conditions have
natural and physical meanings, the periodic boundary conditions can
also be reasonably justified from the physical point of view, when
the vesicle interface $\Gamma$ is sufficient small compared with the
overall physical domain $\Omega$ (cf. \cite{DLW06}). In this
paper, we shall consider the system \eqref{NS}--\eqref{phasefield}
 subject to the periodic boundary conditions (i.e., in torus
$\mathbb{T}^3$):
 \be u(x+e_i) = u(x), \ \
\phi(x+e_i)=\phi(x), \ \ \mbox{for} \ x \in
\partial Q,
  \label{BC}
  \ee
and to the initial conditions
 \be u|_{t=0}=u_0(x), \ \mbox{with} \ \nabla\cdot
u_0=0, \ \int_{Q}u_0dx=0 \ \mbox{and} \ \phi|_{t=0}=\phi_0(x), \ \
\mbox{for} \ x \in Q,
 \label{IC}
 \ee
 where $Q$ is a unit square in $\mathbb{R}^3$.

The main purpose of this paper is to study the existence, regularity
and stability of global strong solutions to the problem
\eqref{NS}--\eqref{IC}. In the subsequent proof, we shall see that,
due to the membrane deformation and the contribution of the
convection term to the phase-field evolution, the problem
\eqref{NS}--\eqref{IC} admits an energy dissipation mechanism (cf.
\eqref{basic energy law} below) that plays a crucial role in
controlling the contribution to the momentum equation of the extra
stress tensor. Comparing with \cite{LYN}, the advantage to work in the periodic setting is that
one can get rid of certain boundary terms when performing
integration by parts. Due to the weak coupling in the phase-field
equation \eqref{phasefield}, which is a gradient flow of the elastic
bending energy under fluid convection, we can derive
uniform-in-time estimate for $H^3$-norm of $\phi$ (cf. Proposition
\ref{comp}) that enables us to derive some specific higher-order
energy inequalities (cf. Lemma \ref{highorder1} and Lemma
\ref{highlargev}) in the spirit of \cite{LL95} for a simplified
nematic crystal system. Based on these higher-order inequalities, we
can show existence and uniqueness of local strong solutions to the
problem \eqref{NS}--\eqref{IC} (cf. Theorem \ref{locstr}), existence
of global strong solutions under properly large viscosity $\mu$ (cf.
Theorem  \ref{theorem in large viscosity case}) and also the
eventual regularity of global weak solutions (cf. Corollary
\ref{evereg}). After a careful exploration of the nonlinear coupling
between the velocity field and membrane deformation, we establish
some regularity criteria for solutions to problem
\eqref{NS}--\eqref{IC} that only involve the velocity field in three dimensional case
(cf. Theorems \ref{theorem on logarithmic criterion}, \ref{theorem
on logarithmic criterion2}), which coincide with the results for
conventional Navier--Stokes equations. This indicates that the
velocity field indeed plays a dominant role in studying regularity
for solutions $(u, \phi)$. Finally, we prove the well-posedness and
stability of global strong solutions when the initial velocity is small and the initial phase function is close to a certain
local minimizer of the elastic energy (cf. Theorem \ref{theorem on
near equilibrium case}), by using a suitable \L
ojasiewicz--Simon type inequality (cf. Lemma \ref{ls}). 

The results
obtained in this paper hold for any given (but fixed) penalty
constants. Since we are working with the penalty formulation to
incorporate the volume and surface area conservation of the vesicle
membrane, these constraints are satisfied only approximately. It would
be interesting to investigate the
evolution system in the Lagrange multiplier formulation (cf.
\cite{DLRW09}), where the volume/area constraints are
satisfied exactly. 
We refer to the recent work \cite{CL11,CL12} for well-posedness results on the single phase-field equation with volume/area constraints, but without coupling with the fluid.

The remaining part of the paper is organized as follows. In Section 2, we
present the functional settings and some preliminary results. In
Section 3, we prove the existence of local strong solutions and
global ones under the large viscosity assumption. In Section 4, we
establish some logarithmic-type regularity criteria for the smooth
solutions only in terms of the velocity field. In Section 5, we
study the well-posedness and stability of global strong solutions
near local minimizers of the elastic energy. In the final Section 6,
the appendix, we provide a formal physical derivation of the
hydrodynamic system \eqref{NS}--\eqref{phasefield} via the
energetic variational approach, and then sketch the proof of the
\L ojasiewicz--Simon type inequality that plays a key role in the
analysis of Section 5.

\section{Preliminaries}\setcounter{equation}{0}
 We recall the well established functional
settings for periodic problems (cf. \cite{Te}):
 \bea
H^m_p(Q)&=&\{u\in H^m(\mathbb{R}^n,\mathbb{R})\ |\ u(x+e_i)=u(x)\},\non\\
\mathbf{H}&=&\{v\in \mathbf{L}^2_p(Q) ,\ \nabla\cdot v=0\},\ \
\text{where}\ \mathbf{L}^2_p(Q)=\mathbf{H}^0_p(Q),
\non\\
\mathbf{V}&=&\{v\in \mathbf{H}^1_p(Q),\ \nabla\cdot v=0\},\non\\
\mathbf{V}'&=&\text{the\ dual space of\ } V.\non
 \eea
For any Banach space $B$, we denote by $\mathbf{B}$ the vector space
$(B)^r$, $r\in \mathbb{N}$, endowed with product norms. For any norm
space $X$, its subspace that consists of functions in $X$ with
zero-mean will be denoted by $\dot X$
 such that $\dot X=\left\{w\in X: \int_{Q} w \,dx=0\right\}.$ We denote the inner
product on $L^2_p(Q)$ (or $\mathbf{L}^2_p(Q))$ as well as $H$ by
$(\cdot,\cdot)$ and the associated norm by $\|\cdot\|$. The space
$H^m_p(Q)$ will be short-handed by $H^m_p$. We denote by $C$ and
$C_i, i=0, 1, \cdots$ generic constants which may depend only on
$\mu, \gamma, Q, \alpha, \beta$ and the initial data $(u_0,
\phi_0)$. Special dependence will be pointed out explicitly in the
text if necessary. Throughout the paper, the Einstein summation
convention will be used. Following \cite{Te}, one can define the
mapping $S$
 \be S u=-\Delta u, \quad  \forall\, u\in D(S):=\{u\in \mathbf{H}, \Delta u\in \mathbf{H}\}=\dot {\mathbf{H}}^2_p\cap
 \mathbf{H}.\label{stokes}
 \ee
 The Stokes operator $S$ can be viewed as an unbounded
positive linear self-adjoint operator on $H$. If $D(S)$ is endowed
with the norm induced by $\dot{\mathbf{L}}^2_p$, then $S$ becomes an
isomorphism from $D(S)$ onto $\mathbf{H}$. More detailed properties
of the operator $S$ can be found in \cite{Te}. We also recall
 the interior elliptic estimate, which states that for bounded domains $U_1\subset\subset U_2$
 there is a constant $C>0$ depending only on $U_1$ and $U_2$ such that
 $\|\phi\|_{H^2(U_1)}\leq C(\|\Delta \phi\|_{L^2(U_2)}+\|\phi\|_{L^2(U_2)})$. In our
 current case under periodic boundary conditions, we can choose $Q'$ to be the union of $Q$ and its
 neighborhood copies. Then we have
 \be
 \|\phi\|_{H^2(Q)}\leq C(\|\Delta \phi\|_{L^2(Q')}+\|\phi\|_{L^2(Q')})= 9C(\|\Delta
 \phi\|_{L^2(Q)}+\|\phi\|_{L^2(Q)}).\label{dh2}
 \ee
 It follows from the periodic boundary condition that $\int_{Q}\nabla\phi\,dx=0$ and $\int_{Q}\Delta\phi\,dx=0$, then we infer from the Poincar\'{e}--Wirtinger inequality that
 \be
  \|\nabla\phi\| + \left|\int_Q \phi dx\right| \approx \|\phi\|_{H^1}, \ \  \|\Delta \phi\|+ \left|\int_Q \phi dx\right| \approx \|\phi\|_{H^2}, \ \  \|\nabla\Delta\phi\|+ \left|\int_Q \phi dx\right| \approx \|\phi\|_{H^3}. \label{from Poincare}
 \ee
A direction calculation yields that the variation of the approximate
elastic energy is given by
 \be
\dfrac{\delta E(\phi)}{\delta\phi}= kg(\phi)+
M_1(A(\phi)-\alpha)+M_2(B(\phi)-\beta)f(\phi), \label{variation of
energy}
  \ee
  where
  \be
   g(\phi)=-\Delta f(\phi)+ \frac{1}{\varepsilon^2}(3\phi^2-1)f(\phi).\non
  \ee
Since we are now dealing with the periodic boundary conditions, the
average of velocity $u$ is conserved.
\begin{lemma}\label{constant velocity mean}
Let $(u, \phi)$ be a solution to the problem \eqref{NS}--\eqref{IC}
on $[0,T]$. It holds
 \be
 \int_Q u(t) dx=\int_Q u_0 dx, \quad \forall\, t\in [0,T]. \label{zerou}
 \ee
\end{lemma}
\begin{proof}
It follows from \eqref{variation of energy} that
 \be
\dfrac{\delta E}{\delta\phi}\nabla\phi= kg(\phi)\nabla \phi +
M_1(A(\phi)-\alpha)\nabla \phi +M_2(B(\phi)-\beta)f(\phi)\nabla \phi
:= I_1+I_2+I_3.\ee Since $A(\phi)$ and $B(\phi)$ are functions only
depending on time, using integration by parts and the periodic
boundary conditions, we deduce that
 \be
 I_2=M_1(A(\phi)-\alpha)\int_Q \nabla \phi dx=0,\non
 \ee
 \bea
 I_3&=& M_2(B(\phi)-\beta)\int_Q f(\phi)\nabla \phi dx\non\\
 &=& M_2(B(\phi)-\beta)\int_Q\Big( -\varepsilon \Delta \phi+ \frac{1}{\varepsilon}(\phi^2-1)\phi\Big)\nabla \phi dx\non\\
 &=& M_2(B(\phi)-\beta)\int_Q \nabla \Big( \frac{\varepsilon}{2}|\nabla \phi|^2+  \frac{1}{4\varepsilon}(\phi^2-1)^2\Big)dx \non\\
 && -M_2(B(\phi)-\beta)\int_Q \nabla \cdot(\nabla \phi \otimes \nabla \phi) dx\non\\
 &=&0,\non
 \eea
 where we have used the fact $\Delta \phi\nabla\phi=\nabla \cdot(\nabla \phi\otimes\nabla \phi)-\frac12\nabla(|\nabla \phi|^2)$.
Finally,
 \bea
  \frac{1}{k}I_1&=& \varepsilon \int_Q \Delta ^2 \phi \nabla \phi dx-\frac{1}{\varepsilon} \int_Q \Delta (\phi ^3-\phi) \nabla \phi dx \non\\
  && +\frac{1}{\varepsilon^2}\int_Q \Big[-\varepsilon \Delta \phi+\frac{1}{\varepsilon} (\phi^2-1)\phi\Big]\nabla (\phi^3-\phi)dx \non\\
  &=& \frac{\varepsilon}{2} \int_Q \nabla |\Delta \phi|^2 dx  -\frac{1}{\varepsilon} \int_Q \Delta (\phi ^3-\phi) \nabla \phi dx\non\\
  &&  -\frac{1}{\varepsilon}\int_Q \Delta\phi \nabla (\phi^3-\phi) dx  +\frac{1}{2\varepsilon^3}\int_Q \nabla[(\phi^3-\phi)^2] dx\non\\
  \non\\
  &=& 0.\non
 \eea
 Thus, we conclude that
 \be
  \int_{Q}\dfrac{\delta E}{\delta\phi}\nabla\phi\, dx=0. \label{zero}
 \ee
 After integrating \eqref{NS} over $Q$, we infer from \eqref{incompressibility}, the periodic boundary condition \eqref{BC} and \eqref{zero} that \eqref{zerou} holds.
\end{proof}
\br By Lemma \ref{constant velocity mean}, if one assumes that the
average of the initial velocity vanishes, i.e.,
 $$\frac{1}{|Q|}\int_Q u_0dx=0,$$
  then we can apply the Poincar\'{e}--Wirtinger inequality to the solution $u$ such that the $\mathbf{H}^1$-norm of $u$ can be controlled by $\|\nabla u\|$. When a flow with non-vanishing average velocity $u$ is considered, as for the single Navier--Stokes equation (cf. \cite{Te}),
  we can introduce the variable  $\tilde{u}= u-\frac{1}{|Q|}\int_Q u dx$ and transform the problem \eqref{NS}--\eqref{IC} into a new one in terms of $\tilde{u}$ and $\phi$. Since $\frac{1}{|Q|}\int_Q u dx$ is a known constant determined by \eqref{zerou}, it is not difficult to verify that our results on existence and uniqueness of weak/strong solutions for the initial velocity with zero mean can be extended to this case with minor modifications. However, results on long-time dynamics in Section 5 are no longer valid, because the velocity will not decay to zero (we also refer to \cite{WXL} for a similar situation for the liquid crystal system).
\er

For the sake of simplicity, in the remaining part of this paper, we
will assume that the average flow vanishes. An important property of
the coupling system \eqref{NS}--\eqref{IC} is that it has a basic
energy law, which indicates the dissipative nature of the system. It
states that the total sum of the kinetic and elastic energy is
dissipated due to viscosity and other possible
regularization/relaxations rates. A formal
 derivation can be carried out by multiplying \eqref{NS}
by $u$, \eqref{phasefield} by $\frac{\delta E(\phi)}{\delta \phi}$,
respectively, and then integrating over $Q$. Consequently, we have

\begin{lemma}[Basic energy law] \label{BEL} Let $(u, \phi)$ be a smooth solution to the problem \eqref{NS}--\eqref{IC},
then the following dissipative energy inequality holds:
 \be
 \dfrac{d}{dt}\Big(\frac{1}{2}\|u(t)\|^2+E(\phi(t))\Big)+\mu\|\nabla u\|^2
 +\gamma\Big\|\dfrac{\delta E}{\delta\phi}\Big\|^2=0,\quad \forall\, t>0. \label{basic energy law}
 \ee
\end{lemma}
 Based on Lemma \ref{BEL}, we can apply the Galerkin method similar to that in \cite{DLL07} to prove the following result on existence and uniqueness of weak solutions
 to the problem \eqref{NS}--\eqref{IC}.

\begin{theorem}[Existence of weak solutions]\label{weak}
Let $n=3$. For any initial datum $(u_0, \phi_0)\in \dot{\mathbf{H}}
\times H^2_p$, $T>0$,   there exists at least one global weak
solution $(u, \phi)$ to the problem \eqref{NS}--\eqref{IC} that
satisfies
 \bea
  &&u \in L^{\infty}(0, T; \dot{\mathbf{H}} ) \cap L^2(0, T;
\dot{\mathbf{V}}); \label{lower bound for velocity}\\
&&\phi \in L^\infty(0, T; H_p^2)\cap L^2(0, T; H_p^4)\cap H^1(0, T;
L_p^2). \label{lower bound for phase}
 \eea
 In addition, the weak solution is unique provided that  $u\in L^8(0, T;
 \mathbf{L}^4_p)$.
\end{theorem}

Besides, we can obtain the following uniform-in-time estimates on
weak solutions from the basic energy
 law:
\begin{proposition}
\label{low} Suppose $n=3$. For any initial data $u_0\in
\dot{\mathbf{H}}$, $\phi_0\in H_p^2$, the corresponding weak
solutions to the problem  \eqref{NS}--\eqref{IC} have the following
uniform estimates
 \be
 \|u(t)\|+\|\phi(t)\|_{H^2}\leq C, \quad \forall\, t\geq 0, \label{unilow}
 \ee
  \be \int_0^{+\infty} \left(\mu\|\nabla u(t)\|^2+\gamma\Big\|\dfrac{\delta E}{\delta\phi}(t)\Big\|^2\right)dt\leq C,
 \label{int}
 \ee
 where $C>0$ is a constant depending on $\|u_0\|,
 \|\phi_0\|_{H^2}$ and coefficients of the system except the viscosity $\mu$.
 \end{proposition}
\begin{proof}
We can derive a weaker version of the basic energy law rigorously
via the Galerkin procedure such that the weak solution $(u, \phi)$
to problem  \eqref{NS}--\eqref{IC} satisfies
 \be \frac12\|u(t)\|^2+ E(\phi(t)) +\int_0^{t} \left(\mu\|\nabla u(s)\|^2+\gamma\Big\| \dfrac{\delta E}{\delta\phi}(s)\Big\|^2\right)ds\leq
\frac12\|u_0\|^2+ E(\phi_0), \quad \forall\, t\geq 0. \non
 \ee
Recalling the definition of $E$, we know $\frac12\|u_0\|^2+ E(0)$
can be estimated by a constant depending on $\|u_0\|,
 \|\phi_0\|_{H^2}$ and coefficients of the system, but not on $\mu$. Thus \eqref{int} holds and $\|u(t)\|$, $E(t)$ are bounded. On the other hand, we infer from the boundedness of $E(t)$ that
 \be
  |A(\phi)|\leq C, \quad |B(\phi)|\leq C, \quad \|f(\phi)\|\leq C,\quad \forall\, t\geq 0.\non
 \ee
 Hence $|\int_Q\phi dx|$ and $\|\nabla \phi\|$ are bounded. Then by the definition of $f(\phi)$ and Sobolev embedding theorems, we can deduce that $\|\phi\|_{H^2}$ is bounded. The proof is complete.
\end{proof}

\section{Existence of strong solutions}\setcounter{equation}{0}

In this section, we study the existence of strong solutions. For
this purpose, it suffices to derive proper higher-order uniform
estimates for the Galerkin approximation of weak solutions and then
pass to limit. We observe that the entire calculation is identical
to that as we work with classical (smooth) solutions to the problem
\eqref{NS}--\eqref{IC}. Thus, for the sake of simplicity, all the
calculations below will be carried out formally for smooth
solutions.

By Sobolev embedding theorems in the three dimensional case, we can derive the following
 estimates that will be frequently used later.
 \begin{lemma}\label{eqes}
 Suppose $n=3$. We have
 $$  \|\nabla \Delta \phi\| \leq
 C \left\|\dfrac{\delta E}{\delta\phi}\right\|^\frac12+C,\quad
 \|\Delta^2 \phi\| \leq  \frac{1}{k\varepsilon} \left\|\dfrac{\delta E}{\delta\phi}\right\|+C,\quad \forall\, \phi\in H^4_p,
 $$
 where $C$ is a constant depending on $\|\phi\|_{H^2}$ and coefficients of the system. Besides,
 $$\|\nabla \Delta^2 \phi\|\leq \frac{1}{k\varepsilon} \left \|\nabla \dfrac{\delta E}{\delta\phi}\right\|+C, \quad \forall\, \phi\in H^5_p,$$
  where $C$ is a constant depending on $\|\phi\|_{H^3}$ and coefficients of the system.
 \end{lemma}
 \begin{proof}
 Recalling \eqref{variation of energy}, we can rewrite $\dfrac{\delta
E(\phi)}{\delta\phi}$ as
 \be \dfrac{\delta
E(\phi)}{\delta\phi} =k\varepsilon\Delta^2 \phi+H(\phi), \non
 \ee
 where
 \bea
H(\phi)&=&   - \frac{k}{\varepsilon}\Delta(\phi^3-\phi)  + \frac{k}{\varepsilon^2}(3\phi^2-1)f(\phi)+ M_1(A(\phi)-\alpha)+M_2(B(\phi)-\beta)f(\phi)\non\\
&=&
-\frac{6k}{\varepsilon}\phi|\nabla\phi|^2-\frac{2k}{\varepsilon}(3\phi^2-1)\Delta\phi+\frac{k}{\varepsilon^3}(3\phi^2-1)(\phi^3-\phi)\non\\
&& +M_1(A(\phi)-\alpha)+M_2(B(\phi)-\beta)f(\phi).\label{v1E}
 \eea
 By the H\"older inequality and Sobolev embedding theorems, we infer that
 \bea
 \| \Delta^2 \phi\|&\leq& \frac{1}{k\varepsilon}\Big\| \dfrac{\delta
E(\phi)}{\delta\phi}\Big\|+ C\|\phi\|_{L^\infty}\|\nabla \phi\|_{\mathbf{L}^4}^2+C\|\phi\|_{L^\infty}^2\|\Delta \phi\|\non\\
&& +C\|\Delta\phi\|+C(\|\phi\|_{L^\infty}^5+1)+ M_1\|\phi\|_{L^1}+M_1\alpha\non\\
&& +CM_2(\beta+ \|\nabla \phi\|^2+C\|\phi\|_{L^4}^4+C)(\|\Delta \phi\|+C\|\phi\|_{L^6}^3+C\|\phi\|)\non\\
&\leq& \frac{1}{k\varepsilon} \left\|\dfrac{\delta
E}{\delta\phi}\right\|+C,\quad \forall \, \phi\in H^4_p,\label{De2}
 \eea
 where $C$ is a constant depending on $\|\phi\|_{H^2}$ and coefficients of the
 system. \\ 
 The estimate for $\|\nabla \Delta \phi\|$ easily follows from \eqref{De2} and the fact $$\|\nabla \Delta \phi\|^2=\left|\int_\Omega \Delta \phi \Delta^2\phi dx\right|\leq C\| \Delta^2 \phi\|.$$
 Concerning the estimate for $\|\nabla \Delta^2 \phi\|$, we just apply $\nabla$ to $\dfrac{\delta
E(\phi)}{\delta\phi}$ and get
\be 
\nabla\dfrac{\delta E(\phi)}{\delta\phi}
=k\varepsilon\nabla\Delta^2 \phi+\nabla{H}(\phi),\non
 \ee
 where
 \bea
\nabla{H}(\phi)&=&-\frac{6k}{\varepsilon}|\nabla\phi|^2\nabla\phi-\frac{12k}{\varepsilon}\phi\nabla\phi\cdot\nabla\nabla\phi
-\frac{12k}{\varepsilon}\phi\nabla\phi\Delta\phi-\frac{2k}{\varepsilon}(3\phi^2-1)\nabla\Delta\phi
\non\\
&&\;+\frac{k}{\varepsilon^3}(3\phi^2-1)^2\nabla\phi+\frac{6k}{\varepsilon^3}\phi(\phi^3-\phi)\nabla\phi
+M_2\big(B(\phi)-\beta \big)\nabla{f}(\phi).  
\non
 \eea 
 Using the
Sobolev embedding theorems, we infer that 
 \bea
\|\nabla\Delta^2\phi\| 
 &\leq&
\frac{1}{k\varepsilon}\left\|\nabla\dfrac{\delta
E(\phi)}{\delta\phi}\right\|
 +C\|\nabla\phi\|_{L^6}^3+C\|\phi\|_{L^\infty}\|\nabla\phi\|_{L^6}\|\nabla^2\phi\|_{L^3}
 \non\\
 && \ \ +C(\|\phi\|_{L^\infty}^2+1)\|\nabla\Delta\phi\|
    +C(\|\phi\|_{L^\infty}^4+\|\phi\|_{L^\infty}^2+1)\|\nabla\phi\|
    \non\\
 &&\ \ 
    +C\Big[\|\nabla\Delta\phi\|+
(\|\phi\|_{L^\infty}^2+1)\|\nabla\phi\| \Big]
 \non\\
 &\leq&
 \frac{1}{k\varepsilon}\left\|\nabla\dfrac{\delta{E}(\phi)}{\delta\phi}\right\|+C(\|\Delta\phi\|^3+1)
 \non\\
 &&\ \ + C\|\nabla^2\phi\|^{\frac12}\big(\|\nabla\Delta\phi\|^{\frac12}+1\big)+C(\|\nabla\Delta\phi\|+1)
  \non\\
&\leq&\frac{1}{k\varepsilon}\left\|\nabla\dfrac{\delta{E}(\phi)}{\delta\phi}\right\|+C,\non
\eea where $C$ is a constant depending on $\|\phi\|_{H^3}$ and
coefficients of the system. The proof is complete.
 \end{proof}
Since our system \eqref{NS}--\eqref{phasefield} contains the
Navier--Stokes equations as a subsystem, in the three dimensional
case, one cannot expect that the weak solutions will become regular
for strictly positive time. But it is worth noting that, due to the
weak coupling in the phase-field equation \eqref{phasefield} that
only one lower order term $u\cdot\nabla\phi$ is involved with $u$ in
the evolution equation, we can first derive certain regularity
results for the phase function $\phi$ and show that it turns out to
be regular for $t>0$.
\begin{lemma} \label{h3phase}
Let $n=3$. For any smooth solution to the problem
\eqref{NS}--\eqref{IC}, it holds that
 \be \dfrac{d}{dt}\|\nabla\Delta\phi\|^2
+k\gamma\varepsilon\|\nabla\Delta^2\phi\|^2\leq C(\|\nabla
u\|^2+1)\|\nabla\Delta\phi\|^2+C(1+\|\nabla u\|^2),\label{dphi3}
 \ee
where $C>0$ is a constant depending on $\|u_0\|,
 \|\phi_0\|_{H^2}$ and coefficients of the system.
\end{lemma}
\begin{proof}
Multiplying \eqref{phasefield} by $-\Delta^3\phi$, integrating over
$Q$, we have
 \be
\frac12\dfrac{d}{dt}\|\nabla\Delta\phi\|^2=-\big(\nabla(u\cdot\nabla\phi),
\nabla\Delta^2\phi\big) +\gamma\Big(\nabla\dfrac{\delta
E}{\delta\phi}, \nabla\Delta^2\phi \Big). \label{phidh3}
 \ee
 Using the lower-order uniform estimates in Proposition \ref{low}, we estimate the first term on the right-hand side of \eqref{phidh3} as follows
 \begin{eqnarray}
&&\big(\nabla(u\cdot\nabla\phi), \nabla\Delta^2\phi\big) \non\\
&\leq&
\frac{k\gamma\varepsilon}{8}\|\nabla\Delta^2\phi\|^2+C\|\nabla
u\cdot\nabla\phi\|^2+C\|u\cdot\nabla^2\phi\|^2  \non\\
&\leq&
\frac{k\gamma\varepsilon}{8}\|\nabla\Delta^2\phi\|^2+C\|\nabla
u\|^2\|\nabla\phi\|_{\mathbf{L}^{\infty}}^2+C\|u\|_{\mathbf{L}^6}^2\|\nabla^2\phi\|_{\mathbf{L}^3}^2
\non\\
&\leq&
\frac{k\gamma\varepsilon}{8}\|\nabla\Delta^2\phi\|^2+C\|\nabla
u\|^2\|\nabla\phi\|_{\mathbf{H}^1}\|\nabla\phi\|_{\mathbf{H}^2}\non\\
&\leq&
\frac{k\gamma\varepsilon}{8}\|\nabla\Delta^2\phi\|^2+C\|\nabla
u\|^2(\|\nabla\Delta\phi\|^2+1).\non
\end{eqnarray}
For the second term, we infer from \eqref{variation of energy} that
\bea &&\gamma\Big(\nabla\dfrac{\delta E}{\delta\phi},
\nabla\Delta^2\phi \Big)\non\\
&=& \gamma\left(\nabla\Big(kg(\phi)+ M_1(A(\phi)-\alpha)+M_2(B(\phi)-\beta)f(\phi)\Big), \nabla\Delta^2\phi \right) \non\\
&=& k\gamma(\nabla \Delta f(\phi) , \nabla \Delta^2 \phi)- \frac{k\gamma}{\varepsilon^2}( \nabla [(3\phi^2-1)f(\phi)], \nabla \Delta^2\phi)\non\\
&& +M_2\gamma(B(\phi)-\beta)(\nabla f(\phi), \nabla \Delta^2\phi)\non\\
&:=& I_1+I_2+I_3, \non
 \eea
 where
 \bea I_1 &\leq&  -k\gamma\varepsilon
\|\nabla\Delta^2\phi\|^2 +\frac{k\gamma}{\varepsilon} \|\nabla
\Delta (\phi^3-\phi)\|\|\nabla \Delta^2\phi\|
\non\\
 &\leq& -\frac{7k\gamma\varepsilon}{8}\|\nabla\Delta^2\phi\|^2+C\|\nabla \Delta \phi\|^2 +C\|\phi\|_{L^\infty}^2\|\nabla \phi\|^2_{\mathbf{L}^6}\|\Delta \phi\|^2_{L^3}\non\\
 &&+ C\|\phi\|_{L^\infty}^4\|\nabla \Delta \phi\|^2+C\|\nabla \phi\|_{\mathbf{L}^6}^6 \non\\
 &\leq& -\frac{7k\gamma\varepsilon}{8}\|\nabla\Delta^2\phi\|^2+C(\|\nabla\Delta\phi\|^2+1),\non
 \eea
 \bea
 I_2&\leq& \frac{k\gamma}{\varepsilon^2}\left(\|6f(\phi) \phi\nabla \phi\|\|\nabla \Delta^2 \phi\|+ \|(3\phi^2-1)\nabla f(\phi)\|\|\nabla \Delta^2 \phi\|\right)
 \non\\
 &\leq& \frac{k\gamma\varepsilon}{8}\|\nabla\Delta^2\phi\|^2+C\|\Delta \phi\|_{L^3}^2\|\nabla \phi\|_{\mathbf{L}^6}^2\|\phi\|_{L^\infty}^2+C\|\nabla \phi\|^2\|\phi\|_{L^\infty}^4\|\phi^2-1\|_{L^\infty}^2\non\\
 &&+ C\|\phi^2-1\|_{L^\infty}^2(\|\nabla \Delta \phi\|^2+ \|\phi^2-1\|_{L^\infty}^2\|\nabla \phi\|^2)\non\\
 &\leq& \frac{k\gamma\varepsilon}{8}\|\nabla\Delta^2\phi\|^2+C(\|\nabla\Delta\phi\|^2+1),\non
 \eea
 \bea
 I_3&\leq&  \frac{k\gamma\varepsilon}{8}\|\nabla\Delta^2\phi\|^2+C(B(\phi)-\beta)^2\|\nabla f(\phi)\|^2\non\\
 &\leq& \frac{k\gamma\varepsilon}{8}\|\nabla\Delta^2\phi\|^2 + C(\|\nabla \phi\|^4+ \|\phi^2-1\|^4+1)(\|\nabla \Delta \phi\|^2+ \|\phi^2-1\|_{L^\infty}^2\|\nabla \phi\|^2)\non\\
 &\leq& \frac{k\gamma\varepsilon}{8}\|\nabla\Delta^2\phi\|^2 + C(\|\nabla\Delta\phi\|^2+1). \non
 \eea
Collecting the above estimates together, we arrive at our conclusion
\eqref{dphi3}.
 \end{proof}

Based on the higher-order differential inequality \eqref{dphi3} for the phase function 
$\phi$, we get

 \begin{proposition} \label{comp}
Let $n=3$.  For any $u_0\in  \dot{\mathbf{H}}$, $\phi_0\in
H_p^2$, the weak solution to the problem \eqref{NS}--\eqref{IC}
satisfies
 \be
 \label{pos-t-regu}
 \|\phi(t)\|_{H^3}\leq C\left(1+\frac1t\right) \ \  \text{and}\ \ \|\nabla \phi(t)\|_{\mathbf{L}^\infty}\leq C\left(1+\frac1t\right), \quad  \forall\, t\,>\,0,
 \ee
 where $C$ is a constant depending on
  $\|u_0\|$, $\|\phi_0\|_{H^2}$ and coefficients of the system.
  Moreover, if we further assume that $\phi_0\in H^3_p$, then
 \be
 \label{phih3}
 \|\phi(t)\|_{H^3}\leq C\ \ \text{and}\ \ \|\nabla \phi(t)\|_{\mathbf{L}^\infty}\leq C, \quad  \forall\, t\,\geq\,0,
 \ee where $C$ is a constant depending on
  $\|u_0\|$, $\|\phi_0\|_{H^3}$ and coefficients of the system.
 \end{proposition}
  \begin{proof}
 We infer from Proposition \ref{low} and Lemma \ref{eqes} that for any
 $r>0$ and $t\geq 0$,
 \bea
  && \sup_{t\geq 0} \int_t^{t+r} \|\nabla \Delta
 \phi(\tau)\|^2d\tau \leq \sup_{t\geq 0} C\int_t^{t+r}\left\|\dfrac{\delta E}{\delta\phi}(\tau)\right\|^2
 d\tau+Cr\non\\
 &\leq& C \int_0^{+\infty}\left\|\dfrac{\delta E}{\delta\phi}(\tau)\right\|^2d\tau+Cr\leq C(1+r),\label{inter1}
 \eea
 \be
 \sup_{t\geq 0} \int_t^{t+r} \|\nabla u(\tau)\|^2d\tau\leq \int_0^{+\infty} \|\nabla u(\tau)\|^2d\tau\leq C.\label{inter2}
 \ee
 Then it follows from \eqref{dphi3} and the uniform Gronwall lemma \cite[Lemma III.1.1]{Te97} that
 \be
 \|\nabla \Delta \phi(t+r)\|^2\leq C\left(1+\frac{1}{r}\right), \quad \forall\, t\geq 0,\  r>0,\label{uge}
 \ee
 which yields \eqref{pos-t-regu}. The estimate for $\|\nabla \phi(t)\|_{\mathbf{L}^\infty}$ follows from the continuous embedding $H^2\hookrightarrow L^\infty$ ($n=3$).

If we further assume that $\phi_0\in H^3_p$, then by the
 standard Gronwall inequality, we see that $ \|\nabla \Delta \phi(t)\|$ is also bounded for $t\in [0,1]$. This combined with \eqref{uge} yields our conclusion.  The proof is complete.
  \end{proof}
 \br \label{independmu}
  We remark that the generic constant
$C$ throughout the proof of Lemma \ref{h3phase} does not depend on
the viscosity $\mu$, thus the uniform bounds for  $\|\phi\|_{H^3}$
obtained in Proposition \ref{comp} is independent of $\mu$.
 \er

Define
 \be
\mathcal{A}(t)=\|\nabla u\|^2(t)+\eta\Big\|\dfrac{\delta
E}{\delta\phi}\Big\|^2(t),  \label{A}
 \ee
where $\eta>0$ is a proper constant to be determined later, which
might depend on $\|u_0\|$, $\|\phi_0\|_{H^3}$ and coefficients of
the system.
\begin{lemma} \label{highorder1}
Let $n=3$.  For any smooth solution to the problem
\eqref{NS}--\eqref{IC}, if
 \be
 \|\phi(t)\|_{H^3}+\|\nabla \phi(t)\|_{\mathbf{L}^\infty}\leq K, \quad \forall\, t\geq 0,\label{K}
 \ee
 then for
  \be
 \eta=\frac{\mu\gamma}{16k\varepsilon K^2}, \label{eta}
 \ee
  the following higher-order energy inequality holds:
 \be
\dfrac{d }{dt}\mathcal{A}(t) +\mu\|\Delta u\|^2+
k\gamma\varepsilon\eta\Big\|\Delta\dfrac{\delta E}{\delta\phi}
\Big\|^2\leq C_\ast(\mathcal{A}^3(t)+\mathcal{A}(t)), \label{higha}
 \ee
 where $C_*$ is a constant depending on
  $\|u_0\|$, $\|\phi_0\|_{H^2}$, $K$ and coefficients of the system.
\end{lemma}
\begin{proof}
By equation \eqref{NS} and the periodic boundary conditions
\eqref{BC}, we see that
 \be \frac12 \frac{d}{dt}\|\nabla u\|^2=-(u_t, \Delta
u)=-\mu\|\Delta u\|^2+(u\cdot\nabla u, \Delta u)-\Big(\dfrac{\delta
E}{\delta\phi}\nabla\phi,\Delta u \Big)\label{duu}
 \ee
 Using the uniform estimates \eqref{unilow} and \eqref{K}, the right-hand side of \eqref{duu} can be estimated as follows
  \bea (u\cdot\nabla u, \Delta u) &\leq& \frac{\mu}{16}\|\Delta
u\|^2+C\|u\cdot\nabla u\|^2\non\\
& \leq & \frac{\mu}{16}\|\Delta
u\|^2+C\|u\|_{\mathbf{L}^{\infty}}^2\|\nabla u\|^2 \non\\
&\leq& \frac{\mu}{16}\|\Delta u\|^2+C(\|\nabla u\|\|\Delta
u\|+\|\nabla u\|^2)\|\nabla u\|^2
\non\\
&\leq& \frac{\mu}{8}\|\Delta u\|^2+C(\|\nabla u\|^6+\|\nabla u\|^2),
\label{term 1}
 \eea
 \bea
-\Big( \dfrac{\delta E}{\delta\phi}\nabla\phi, \Delta u\Big) &\leq&
\frac{\mu}{8}\|\Delta u\|^2+\frac{2}{\mu}\Big\|\dfrac{\delta
E}{\delta\phi}\Big\|^2\|\nabla\phi\|_{\mathbf{L}^\infty}^2 \leq
\frac{\mu}{8}\|\Delta u\|^2+C\Big\|\dfrac{\delta
E}{\delta\phi}\Big\|^2. \label{term 2}
 \eea
On the other hand, using integration by parts, we obtain from
\eqref{phasefield} that
 \bea && \frac12\dfrac{d}{dt}\Big\|\dfrac{\delta
E}{\delta\phi}\Big\|^2=\Big(\dfrac{\delta E}{\delta\phi},
\frac{\partial}{\partial
t}\dfrac{\delta E}{\delta\phi}\Big) \non\\
&=&k\varepsilon\Big(\Delta^2\phi_t, \dfrac{\delta
E}{\delta\phi}\Big) -\frac{k}{\varepsilon}\left( \partial_t[\Delta
(\phi^3-\phi)], \dfrac{\delta E}{\delta\phi}\right)
+\frac{6k}{\varepsilon^2}\left(\phi f(\phi)\phi_t,\dfrac{\delta
E}{\delta\phi}\right)
\non\\
&& +\frac{k}{\varepsilon^2}\left((3\phi^2-1)\partial_t
f(\phi),\dfrac{\delta
E}{\delta\phi}\right)+M_1\frac{d}{dt}A(\phi)\int_Q \dfrac{\delta
E}{\delta\phi} dx
\non\\
&& + M_2\frac{d}{dt}B(\phi)\left(f(\phi), \dfrac{\delta
E}{\delta\phi}\right)+ M_2(B(\phi)-\beta)\left(\partial_t f(\phi),
\dfrac{\delta E}{\delta\phi}\right)
 \non\\
&=& \sum_{i=1}^7 J_i. \label{dtEphi}
 \eea
  The first term $J_1$ can be estimated as follows
\bea J_1&=&-k\varepsilon\gamma \left\|\Delta \dfrac{\delta
E}{\delta\phi}\right\|^2-k\varepsilon\left(\Delta(u\cdot\nabla
\phi), \Delta \dfrac{\delta E}{\delta\phi}\right)
\non\\
&\leq& -\frac{7 k\varepsilon\gamma}{8}\Big\|\Delta\dfrac{\delta
E}{\delta\phi}\Big\|^2+\frac{2k\varepsilon}{\gamma}\big\|\Delta(u\cdot\nabla\phi)\big\|^2
\non\\
&\leq& -\frac{7 k\varepsilon\gamma}{8}\Big\|\Delta\dfrac{\delta
E}{\delta\phi}\Big\|^2+\frac{2k\varepsilon}{\gamma}\big(\|\Delta
u\cdot\nabla\phi\big\|^2+2\|\nabla u\cdot\nabla^2\phi\|^2+\|
u\cdot\nabla\Delta\phi\|^2 \big)   \non\\
&\leq& -\frac{7 k\varepsilon\gamma}{8}\Big\|\Delta\dfrac{\delta
E}{\delta\phi}\Big\|^2+\frac{2k\varepsilon}{\gamma}\big(\|\Delta
u\|^2\|\nabla\phi\|_{\mathbf{L}^\infty}^2+2\|\nabla
u\|^2_{\mathbf{L}^3}\|\phi\|^2_{W^{2,6}}+\|u\|_{\mathbf{L}^\infty}^2\|\nabla\Delta\phi\|^2 \big) \non\\
&\leq& -\frac{7 k\varepsilon\gamma}{8}\Big\|\Delta\dfrac{\delta
E}{\delta\phi}\Big\|^2+\frac{2k\varepsilon K^2}{\gamma}\|\Delta
u\|^2+CK^2\|\Delta u\|\|\nabla u\|\non\\
&\leq& -\frac{7 k\varepsilon\gamma}{8}\Big\|\Delta\dfrac{\delta
E}{\delta\phi}\Big\|^2+\frac{4k\varepsilon K^2}{\gamma}\|\Delta
u\|^2+C\|\nabla u\|^2. \label{term3}
 \eea
Then for $J_2, J_3, J_4$, a direct computation yields that
 \bea
  && J_2+J_3+J_4\non\\
  &=& -\frac{6k}{\varepsilon}\left( \partial_t(|\nabla \phi|^2\phi), \dfrac{\delta E}{\delta\phi}\right)-\frac{3k}{\varepsilon}\left(\partial_t(\phi^2\Delta \phi), \dfrac{\delta E}{\delta\phi}\right)+ \frac{k}{\varepsilon}\left(\Delta \phi_t, \dfrac{\delta E}{\delta\phi}\right)\non\\
 && -\frac{6k}{\varepsilon}\left(\phi\Delta \phi \phi_t, \dfrac{\delta E}{\delta\phi}\right)+\frac{6k}{\varepsilon^3}\left((\phi^4-\phi^2)\phi_t, \dfrac{\delta E}{\delta\phi}\right)\non\\
  &&-\frac{k}{\varepsilon}\left((3\phi^2-1)\Delta \phi_t, \dfrac{\delta E}{\delta\phi}\right)+\frac{k}{\varepsilon^3}\left((3\phi^2-1)^2\phi_t, \dfrac{\delta E}{\delta\phi}\right)
\non\\
  &=&-\frac{6k}{\varepsilon}\left( |\nabla \phi|^2\phi_t, \dfrac{\delta E}{\delta\phi}\right)-\frac{12k}{\varepsilon}\left(\phi\phi_t\Delta \phi, \dfrac{\delta E}{\delta\phi}\right)\non\\
  &&
   +\frac{k}{\varepsilon}\left( [2(1-3\phi^2)\Delta  \phi_t-6\nabla \phi^2\cdot \nabla \phi_t],  \dfrac{\delta E}{\delta\phi}\right)\non\\
   && +\frac{k}{\varepsilon^3}\left((15\phi^4-12\phi^2+1)\phi_t, \dfrac{\delta E}{\delta\phi}\right)\non\\
  &:=& J_{2a}+J_{2b}+J_{2c}+J_{2d}.\non
 \eea
 Then we have
\bea J_{2a} &\leq&
C\|\nabla\phi\|_{\mathbf{L}^\infty}^2\|\phi_t\|\Big\|\dfrac{\delta
E}{\delta\phi}\Big\|\leq C\|u\cdot\nabla\phi\|\Big\|\dfrac{\delta
E}{\delta\phi}\Big\|+C\Big\|\dfrac{\delta
E}{\delta\phi}\Big\|^2 \non\\
&\leq&  C\Big\|\dfrac{\delta
E}{\delta\phi}\Big\|^2+C\|u\|_{\mathbf{L}^6}^2\|\nabla\phi\|_{\mathbf{L}^3}^2\non\\
&\leq& C\Big\|\dfrac{\delta E}{\delta\phi}\Big\|^2+C\|\nabla u\|^2,
\non
 \eea
 \bea
 J_{2b} &\leq& C\|\phi\|_{L^\infty}\|\phi_t\|\|\Delta\phi\|_{L^6}\Big\|\dfrac{\delta
E}{\delta\phi}\Big\|_{L^3} \non\\
&\leq& C\left(\|u\|_{\mathbf{L}^6}\|\nabla
\phi\|_{\mathbf{L}^3}+\Big\|\dfrac{\delta
E}{\delta\phi}\Big\|\right)\left(\Big\|\Delta\dfrac{\delta
E}{\delta\phi}\Big\|^{\frac14}\Big\|\dfrac{\delta
E}{\delta\phi}\Big\|^{\frac34}+\Big\|\dfrac{\delta
E}{\delta\phi}\Big\|\right)  \non\\
&\leq& \frac{ k\varepsilon\gamma}{16}\Big\|\Delta \dfrac{\delta
E}{\delta\phi}\Big\|^2+C\Big\|\dfrac{\delta
E}{\delta\phi}\Big\|^2+C\|\nabla u\|^2,\non
 \eea
\bea
 J_{2c}
 &=& \frac{2k}{\varepsilon}\left((3\phi^2-1)\nabla  \phi_t, \nabla \dfrac{\delta E}{\delta\phi}\right)\non\\
 &\leq& C(\|\phi\|_{L^\infty}^2+1)\|\nabla\phi_t\|\Big\|\nabla\dfrac{\delta
E}{\delta\phi}\Big\| \leq C\Big\|\nabla\dfrac{\delta
E}{\delta\phi}\Big\|^2+C\|\nabla(u\cdot\nabla\phi)\|^2 \non\\
&\leq& C\left(\Big\|\Delta\dfrac{\delta
E}{\delta\phi}\Big\|\Big\|\dfrac{\delta
E}{\delta\phi}\Big\|+\Big\|\dfrac{\delta
E}{\delta\phi}\Big\|^2\right)+C\|\nabla u\|^2\|\nabla
\phi\|_{\mathbf{L}^\infty}^2+C\|u\|_{\mathbf{L}^6}^2\|\phi\|^2_{W^{2,3}}
\non\\
&\leq&\frac{ k\varepsilon\gamma}{16}\Big\|\Delta \dfrac{\delta
E}{\delta\phi}\Big\|^2+C\Big\|\dfrac{\delta
E}{\delta\phi}\Big\|^2+C\|\nabla u\|^2,\non
 \eea
 \bea
 J_{2d}&\leq& C(\|\phi\|_{L^\infty}^4+1)\|\phi_t\|\left\| \dfrac{\delta E}{\delta\phi}\right\|\leq C\Big\|\dfrac{\delta
E}{\delta\phi}\Big\|^2+C\|\nabla u\|^2.\non
 \eea
 Hence, we obtain
 \be
 J_2+J_3+J_4\leq \frac{ k\varepsilon\gamma}{8}\Big\|\Delta \dfrac{\delta
E}{\delta\phi}\Big\|^2+C\Big\|\dfrac{\delta
E}{\delta\phi}\Big\|^2+C\|\nabla u\|^2.\non
 \ee
 For the remaining terms, we have
 \be
 J_5 \leq C\|\phi_t\|\Big\|\dfrac{\delta E}{\delta\phi}\Big\|
\leq C\Big\|\dfrac{\delta E}{\delta\phi}\Big\|^2+C\|\nabla
u\|^2,\non
 \ee
 \bea
  J_6 &\leq& C(\|\nabla\phi\|\|\nabla\phi_t\|+\|\phi^3-\phi\|\|\phi_t\|)\|f(\phi)\|\left\| \dfrac{\delta E}{\delta\phi}\right\|
\non\\
&\leq& C(\|\phi_t\|+\|\nabla\phi_t\|)\Big\|\dfrac{\delta
E}{\delta\phi}\Big\| \non\\
&\leq& \frac{ k\varepsilon\gamma}{8}\Big\|\Delta \dfrac{\delta
E}{\delta\phi}\Big\|^2+C\Big\|\dfrac{\delta
E}{\delta\phi}\Big\|^2+C\|\nabla u\|^2, \non
 \eea
 \bea
 J_7 &=& -M_2\varepsilon(B(\phi)-\beta)\Big(\Delta\phi_t, \dfrac{\delta
E}{\delta\phi}\Big)+\frac{M_2}{\varepsilon}(B(\phi)-\beta)\Big((3\phi^2-1)\phi_t,
\dfrac{\delta E}{\delta\phi}\Big)   \non\\
&=& -M_2\varepsilon(B(\phi)-\beta)\Big(\phi_t, \Delta\dfrac{\delta
E}{\delta\phi}\Big)+\frac{M_2}{\varepsilon}(B(\phi)-\beta)\Big((3\phi^2-1)\phi_t,
\dfrac{\delta E}{\delta\phi}\Big)  \non\\
&\leq& C\|\phi_t\|\Big\|\Delta\dfrac{\delta
E}{\delta\phi}\Big\|+C\|\phi_t\|\Big\|\dfrac{\delta
E}{\delta\phi}\Big\|  \non\\
&\leq& \frac{ k\varepsilon\gamma}{8}\Big\|\Delta \dfrac{\delta
E}{\delta\phi}\Big\|^2+C\Big\|\dfrac{\delta
E}{\delta\phi}\Big\|^2+C\|\nabla u\|^2. \non
 \eea
Collecting  the above estimates, we deduce that
 \be  \frac12\dfrac{d}{dt}\Big\|\dfrac{\delta
E}{\delta\phi}\Big\|^2+ \frac{ k\varepsilon\gamma}{2}\Big\|\Delta
\dfrac{\delta E}{\delta\phi}\Big\|^2\leq \frac{4k\varepsilon
K^2}{\gamma}\|\Delta u\|^2+C\Big\|\dfrac{\delta
E}{\delta\phi}\Big\|^2+C\|\nabla u\|^2. \label{dtEphi1}
 \ee
 Now multiplying \eqref{dtEphi1} by $
 \eta=\dfrac{\mu\gamma}{16k\varepsilon K^2}$ and adding the result to \eqref{duu}, we obtain \eqref{higha}. The proof is complete.
\end{proof}

\begin{theorem}[Local strong solution]\label{locstr}
Let $n=3$. For any initial datum $(u_0, \phi_0)\in \dot{\mathbf{V}}
\times H^4_p(Q)$, there exists $T_0\in (0,+\infty)$ such that the
problem \eqref{NS}--\eqref{IC} admits a unique strong solution $(u,
\phi)$ satisfying
 \bea
  &&u \in L^{\infty}(0, T_0; \dot{\mathbf{V}} ) \cap L^2(0, T_0;
 \mathbf{H}^2); \\
&&\phi \in L^\infty(0, T_0; H_p^4)\cap L^2(0, T_0; H_p^6)\cap H^1(0,
T_0; H^2_p).
 \eea
\end{theorem}
\begin{proof}
It follows from Proposition \ref{comp} that the assumption \eqref{K}
in Lemma \ref{highorder1} is satisfied and $K$ is a constant
depending on
  $\|u_0\|$, $\|\phi_0\|_{H^3}$ and coefficients of the system. Consequently, \eqref{higha} holds with $\eta$ and $C_*$ depending on
  $\|u_0\|$, $\|\phi_0\|_{H^3}$ and coefficients of the system. A standard
argument in ODE theory yields that there exists a
$T_0=T_0(\mathcal{A}(0), C_*)\in (0, +\infty)$ such that
$\mathcal{A}(t)$ is bounded on $[0, T_0]$. The bound only depends on
$T_0$, $\mathcal{A}(0)$ and $C_*$. This fact together with the
lower-order estimates in Proposition \ref{low} and the Galerkin
scheme similar to that in \cite{DLL07} implies the
 existence of a local strong solution to the problem
\eqref{NS}--\eqref{IC} in the time interval $[0,T_0]$. Since $u\in
L^{\infty}(0, T_0; \dot{\mathbf{V}} )\subset L^8(0, T_0;
\mathbf{L}^4_p)$, uniqueness of the local strong solution follows
from Theorem \ref{weak}. The proof is complete.
\end{proof}

In general, we cannot expect existence of global strong solutions to
the problem \eqref{NS}--\eqref{IC} for arbitrary initial data in
$\dot{\mathbf{V}}\times H^4_p$, due to the difficulty from the $3D$ Navier--Stokes equations. However, if we
assume that the fluid viscosity $\mu$ is properly large, then problem
\eqref{NS}--\eqref{IC} will admit a unique global strong solution
that is uniformly bounded in $\mathbf{H}^1\times H^4$ on
$[0,+\infty)$. To verify this point, we first derive an alternative
higher-order differential inequality.

\begin{lemma} \label{highlargev}
\noindent Let $n=3$. For arbitrary $\mu_0>0$, if $\mu\geq \mu_0>0$,
and \eqref{K} is satisfied, then choosing the parameter $\eta$ in
$\mathcal{A}(t)$ to be
  \be
 \eta'=\frac{\mu_0\gamma}{16k\varepsilon K^2}, \label{eta1}
  \ee
the following inequality holds for the
 smooth solution $(u,\phi)$ to the problem
\eqref{NS}--\eqref{IC}
 \be \dfrac{d }{dt}\mathcal{A}(t)+\left(\mu-\mu^\frac12\mathcal{A}(t)\right)\|\Delta
u\|^2+k\varepsilon\gamma\eta'\Big\|\Delta\dfrac{\delta
E}{\delta\phi} \Big\|^2\leq C'\mathcal{A}(t),
 \label{3D case 1}
\ee where $C'$ is a constant depending on
  $\|u_0\|$, $\|\phi_0\|_{H^2}$, $K$, $\mu_0$ and coefficients of the system but except $\mu$.
\end{lemma}

\begin{proof}

We only need to refine the estimate \eqref{term 1} in the proof of
Lemma \ref{highorder1} such that
 \bea (u\cdot\nabla u, \Delta
u) &\leq& \frac{\mu}{8}\|\Delta
u\|^2+\frac{2}{\mu}\|u\|_{\mathbf{L}^\infty}^2\|\nabla u\|^2\non\\
& \leq& \frac{\mu}{8}\|\Delta
u\|^2+\frac{C}{\mu}(\|u\|^{\frac12}\|\Delta
u\|^{\frac32}+\|u\|^2)\|\nabla u\|^2  \non\\
&\leq& \frac{\mu}{8}\|\Delta u\|^2+\frac{\mu^\frac12}{2}\|\nabla
u\|^2\|\Delta
u\|^2+C\left({\mu}^{-\frac{11}{2}}+\mu^{-1}\right)\|\nabla u\|^2.
\label{refined term 1}
 \eea
 Since $\mu\geq \mu_0$, we can choose $\eta$ in $\mathcal{A}(t)$ to be
  $\eta'=\frac{\mu_0\gamma}{16k\varepsilon K^2}$. Combining \eqref{refined term 1} with estimates for the other terms in the proof of Lemma \ref{highorder1}, we can easily conclude \eqref{3D case 1} with our choice of $\eta'$. The proof is complete.
\end{proof}

\bt [Global strong solution under large viscosity]
 \label{theorem in large viscosity case} Let $n=3$. For any initial data $(u_0,
d_0)\in \dot{\mathbf{V}}\times H^4_p$, if $\mu$ is sufficiently
large such that \eqref{large mu} is satisfied (see below),
  then the problem \eqref{NS}--\eqref{IC} admits a unique
  global strong solution.
 \et
 \begin{proof}
We infer from \eqref{int} and the choice of $\eta'$ in Lemma
\ref{highlargev} that
 \be \sup_{t\geq0}\int_t^{t+1}\mathcal{A}(\tau)
d\tau \leq \int_0^{+\infty} \mathcal{A}(t) dt \leq M,\non
 \ee
  where $M$ is a constant depending on
  $\|u_0\|$, $\|\phi_0\|_{H^3}$, $\mu_0$ and coefficients of the system but except $\mu$.
If the viscosity $\mu$ satisfies the following relation
 \be \mu^\frac12\geq \mathcal{A}(0)+C'M+4M+\mu_0^\frac12,\label{large mu}
 \ee
by applying the same idea as in \cite[Section 4]{LL95}, we can
deduce from \eqref{3D case 1} that
 $ \mathcal{A}(t)$ is uniformly bounded such that
 $$\mathcal{A}(t)\leq \mu^\frac12, \quad \forall \, t\geq 0.$$
 Based on the uniform-in-time estimates and the Galerkin scheme,
 we are able to prove the existence and uniqueness of the global strong solution to the problem \eqref{NS}--\eqref{IC}.
We leave the details to interested readers.
 \end{proof}

\section{Regularity criteria}\setcounter{equation}{0}
In this section, we are going to establish some regularity criteria
for solutions to the problem \eqref{NS}--\eqref{IC} in the three
dimensional case. These criteria only involve the velocity field,
which indicate that in spite of the nonlinear coupling between the
equations for  velocity field and the phase function, the velocity
field indeed plays a dominant role in regularity for solutions to the
system \eqref{NS}--\eqref{phasefield}, just as the decoupled
incompressible Navier--Stokes equations (cf. e.g., \cite{Be, Se}).

First, we provide a result on regularity criteria in terms of the
velocity $u$ (cf. \cite{Se}) or its gradient $\nabla u$ (cf. \cite{Be}).

\begin{theorem} \label{theorem on logarithmic criterion}
Suppose $n=3$. For $(u_0, \phi_0) \in \dot{\mathbf{V}} \times
H^4_p$, let $(u(t), \phi(t))$ be a local smooth solution to the
problem \eqref{NS}--\eqref{IC} on $[0,T)$ for some $0<T<+\infty$.
Suppose that one of the following conditions holds,
 \begin{itemize}
 \item[(i)] $ \displaystyle{\int_{0}^{T}} \|\nabla u(t)\|_{\mathbf{L}^p}^s dt < +\infty$, for $\displaystyle{\frac{3}{p}+\frac{2}{s} \leq 2}$, $\displaystyle{ \frac32<p\leq +\infty}$,
 \item[(ii)] $ \displaystyle{\int_{0}^{T}\|u(t)\|_{\mathbf{L}^p}^s dt < +\infty}$, for $\displaystyle{\frac{3}{p}+\frac{2}{s} \leq 1}$, $3<p\leq +\infty$.
 \end{itemize}
 Then $(u, \phi)$ can be extended beyond $T$.
\end{theorem}
\begin{proof}
We keep in mind that the uniform estimates \eqref{unilow} and
\eqref{phih3} are satisfied for $t\geq 0$.
Suppose that (i) is satisfied.  For $p> \frac32$,  we estimate
\eqref{term 1} as follows
 \bea (u\cdot\nabla u, \Delta
u)&=&-\int_{Q}\nabla_ku_j\nabla_ju_i\nabla_ku_i dx - \int_Q u_j\nabla_j(\nabla_ku_i) \nabla_ku_idx\non\\
&=&-\int_{Q}\nabla_ku_j\nabla_ju_i\nabla_ku_i dx\non\\
&\leq& C\|\nabla
u\|_{\mathbf{L}^{p}}\|\nabla u\|_{\mathbf{L}^{_\frac{2p}{p-1}}}^2 \non\\
&\leq& C\|\nabla u\|_{\mathbf{L}^{p}}\Big(\|\nabla
u\|^{\frac{2p-3}{p}}\|\Delta u\|^{\frac{3}{p}}+\|\nabla u\|^2 \Big)
\non\\
&\leq& \frac{\mu}{8}\|\Delta u\|^2+C\left(\|\nabla
u\|_{\mathbf{L}^{p}}+\|\nabla
u\|_{\mathbf{L}^p}^{\frac{2p}{2p-3}}\right)\|\nabla u\|^2.
\label{uuu}
 \eea
Combining \eqref{uuu} with the other estimates in the proof of Lemma
\ref{highorder1},  we obtain that
 \be \dfrac{d }{dt}\mathcal{A}(t) +\mu\|\Delta
u\|^2+ k\gamma\varepsilon\eta \Big\|\Delta\dfrac{\delta
E}{\delta\phi} \Big\|^2 \leq C\left(1+\|\nabla
u\|_{\mathbf{L}^p}^{\frac{2p}{2p-3}}\right)\mathcal{A}(t).
\label{rehighorder}
 \ee
 Then by the Gronwall inequality, we see that $\mathcal{A}(t)\leq C_T$ for $t\in [0,T]$, which implies that the $\mathbf{H}^1\times H^4$ norm of the strong solution $(u, \phi)$ is bounded on interval [0,T]. This yields that $[0, T )$ cannot be the maximal interval of existence, and the solution $(u, \phi)$ can be extended
beyond $T$.

 Next, we suppose that (ii) is satisfied. For $p>3$, we estimate \eqref{term
1} in an alternative way such that
 \bea (u\cdot\nabla u, \Delta
u)&\leq& C\|u\|_{\mathbf{L}^{p}}\|\nabla
u\|_{\mathbf{L}^\frac{2p}{p-2}}\|\Delta u\|\non\\
&\leq& 
C\|u\|_{\mathbf{L}^{p}}\|\Delta u\| \left(\|\Delta
u\|^\frac{3}{p}\|\nabla u\|^\frac{p-3}{p}+\|\nabla u\|\right)
\non\\
&\leq& \frac{\mu}{8}\|\Delta
u\|^2+C\left(\|u\|^\frac{2p}{p-3}_{\mathbf{L}^p}+1\right)\|\nabla
u\|^2.
 \eea
  then by the Gronwall inequality, $\mathcal{A}(t)\leq C_T$ for $t\in [0,T]$, which again yields our conclusion. The proof is complete.
\end{proof}

As for the conventional Navier--Stokes equations (see, for instance,
\cite{ZL,ZG,CV}), we can improve the results in Theorem \ref{theorem
on logarithmic criterion} and obtain some logarithmical-type
regularity criteria for our phase-field Navier--Stokes system
\eqref{NS}--\eqref{IC}.
 \begin{theorem} \label{theorem on logarithmic criterion2}
Suppose $n=3$. For $(u_0, \phi_0) \in (\dot{\mathbf{V}}\cap
\mathbf{H}^2_p) \times H^5_p$, let $(u, \phi)$ be a local smooth
solution to the problem \eqref{NS}--\eqref{IC} on $[0,T)$ for some
$0<T<+\infty$. If one of the following conditions holds,
 \begin{itemize}
 \item[(i)] \be \int_{0}^{T} \dfrac{\|\nabla u(t)\|_{\mathbf{L}^p}^s}{1+\ln(e+\|\nabla
u(t)\|_{\mathbf{L}^p})}dt < +\infty, \ \ \ \mbox{for} \
\frac{3}{p}+\frac{2}{s} \leq 2, \ \ \frac32 \leq p \leq 6,
\label{integrability condition1}
 \ee
 \item[(ii)]
 \be
 \int_{0}^{T} \dfrac{\|u(t)\|_{\mathbf{L}^p}^s}{1+\ln(e+\|u(t)\|_{\mathbf{L}^\infty})}dt < +\infty, \ \ \ \mbox{for} \
\frac{3}{p}+\frac{2}{s} \leq 1, \ \ 3<p \leq +\infty,
\label{integrability condition2}
 \ee
 \end{itemize}
 then $(u, \phi)$ can be extended beyond $T$.
\end{theorem}
\begin{proof}
We recall that uniform estimates \eqref{unilow} and \eqref{phih3}
still hold for $t\geq 0$.

\textbf{Case} (i). Suppose that \eqref{integrability condition1} is
satisfied. Applying $\Delta$ to both sides of equation \eqref{NS},
multiplying the resultant with $\Delta u$ and integrating over $Q$,
we get
 \be \frac12\frac{d}{dt}\|\Delta u\|^2+\mu\|\nabla\Delta
u\|^2=\left(\nabla (u\cdot\nabla u), \nabla \Delta u \right)
-\left(\nabla\Big(\dfrac{\delta E}{\delta\phi}\nabla\phi\Big),
\nabla\Delta u \right):=I_1+I_2, \label{time derivative for laplace
u}
 \ee
 where
 \bea I_1&\leq& C\|\nabla \Delta u\|(\|\nabla u\|_{\mathbf{L}^4}^2+\|u\|_{\mathbf{L}^\infty}\|\Delta u\|) \non\\
&\leq& C\|\nabla \Delta u\|\Big(\|\nabla u\|_{\mathbf{L}^4}^2+ \|u\|_{\mathbf{H}^{2}}^\frac12\|u\|_{\mathbf{H}^1}^\frac12\|\Delta u\|\Big)\non\\
&\leq& C\|\nabla \Delta u\|\left(\|\nabla \Delta u\|^\frac12\|\nabla u\|^\frac32+\|\nabla \Delta u\|^\frac34\|\nabla u\|^\frac54+\|\nabla u\|^2\right)\non\\
&\leq& \frac{\mu}{8}\|\nabla\Delta u\|^2+C\big(\|\nabla
u\|^{10}+\|\nabla u\|^{4} \big). \non
 \eea
 Estimate for the term $I_2$ will be postponed. On the other hand, we deduce from equation \eqref{phasefield} that
\bea &&\frac{1}{2}\frac{d}{dt}\Big\|\nabla\dfrac{\delta
E}{\delta\phi}\Big\|^2 = -\left(\frac{\partial}{\partial
t}\dfrac{\delta
E}{\delta\phi}, \Delta\dfrac{\delta E}{\delta\phi} \right) \non\\
&=&-k\varepsilon\Big(\Delta^2\phi_t, \Delta \dfrac{\delta
E}{\delta\phi}\Big)+\frac{k}{\varepsilon}\left( \partial_t[\Delta
(\phi^3-\phi)], \Delta\dfrac{\delta E}{\delta\phi}\right)
-\frac{6k}{\varepsilon^2}\left(\phi
f(\phi)\phi_t,\Delta\dfrac{\delta E}{\delta\phi}\right)
\non\\
&& -\frac{k}{\varepsilon^2}\left((3\phi^2-1)\partial_t f(\phi),\Delta\dfrac{\delta E}{\delta\phi}\right)- M_2\frac{d}{dt}B(\phi)\left(f(\phi), \Delta\dfrac{\delta E}{\delta\phi}\right)\non\\
&& - M_2(B(\phi)-\beta)\left(\partial_t f(\phi),\Delta \dfrac{\delta
E}{\delta\phi}\right)
 \non\\
&=& \sum_{i=1}^6 J'_i. \label{tdphih5}
 \eea
 The term $J_1'$ can be estimated in the following way:
\bea J_1'&=& -k\varepsilon\gamma\Big\|\nabla\Delta\dfrac{\delta
E}{\delta\phi}\Big\|^2-k\varepsilon
\Big(\nabla\Delta(u\cdot\nabla\phi), \nabla \Delta\dfrac{\delta
E}{\delta\phi} \Big) \non\\
&\leq&-\frac{7k\varepsilon\gamma}{8}\Big\|\nabla\Delta\dfrac{\delta
E}{\delta\phi}\Big\|^2+\frac{2k\varepsilon}{\gamma}\big\|\nabla\Delta(u\cdot\nabla\phi)
\big\|^2 \non\\
&\leq&-\frac{7k\varepsilon\gamma}{8}\Big\|\nabla\Delta\dfrac{\delta
E}{\delta\phi}\Big\|^2+C\int_Q(\nabla_j\nabla_i u_k\nabla_i\nabla_k \phi )^2dx+C\int_Q (\nabla_i u_k\nabla_j\nabla_i\nabla_k \phi)^2dx\non\\
&& +C\int_Q(\nabla_j\nabla_i\nabla_i u_k \nabla_k\phi)^2dx+ C\int_Q (\nabla_i\nabla_i u_k \nabla_j\nabla_k \phi)^2dx\non\\
&&+C\int_Q(\nabla_j u_k \nabla_i\nabla_i \nabla_k \phi)^2 dx+C\int_Q(u_k\nabla_j\nabla_i\nabla_i\nabla_k \phi)^2 dx\non\\
&:=& -\frac{7k\varepsilon\gamma}{8}\Big\|\nabla\Delta\dfrac{\delta
E}{\delta\phi}\Big\|^2+ J'_{1a}+...+J'_{1f},\non
 \eea
 where
 \bea
 J'_{1a}+J'_{1d}
 &\leq&
C\|\nabla u\|_{\mathbf{W}^{1,3}}^2\|\phi\|_{W^{2,6}}^2 \leq
C\big(\|\nabla\Delta u\|^\frac32\|\nabla
u\|^\frac12+\|\nabla u\|^2\big) \non\\
 &\leq& \frac13\|\nabla\Delta u\|^2+C\|\nabla u\|^2,
 \non\\
   J'_{1b}+J'_{1e}
   &\leq&
 C\|\nabla u\|_{\mathbf{L}^\infty}^2\|\phi\|_{H^3}^2
 \leq C\big(\|\nabla\Delta u\|^{\frac32}\|\nabla
 u\|^{\frac12}+\|\nabla u\|^2\big)
 \non\\
 &\leq& \frac13\|\nabla\Delta u\|^2+C\|\nabla u\|^2,
 \non\\
  J'_{1c}
  &\leq& C\|\nabla\Delta
u\|^2\|\nabla\phi\|_{\mathbf{L}^\infty}^2 \leq C_1\|\nabla\Delta
u\|^2,
 \non\\
  J'_{1f}
  &\leq&
C\|u\|_{\mathbf{L}^\infty}^2\|\phi\|_{H^4}^2 \leq
C\|u\|_{\mathbf{L}^\infty}^2 \left(\Big\|\dfrac{\delta
E}{\delta\phi}\Big\|^2+1  \right)  \non\\
  &\leq& C(\|
u\|\|\nabla\Delta u\|+\|u\|^2)\Big\|\dfrac{\delta
E}{\delta\phi}\Big\|^2+C\|
\nabla u\|^\frac32\|\nabla \Delta u\|^\frac12 \non\\
 &\leq& \frac13\|\nabla\Delta u\|^2+C\left(\Big\|\dfrac{\delta
E}{\delta\phi}\Big\|^4+\Big\|\dfrac{\delta
E}{\delta\phi}\Big\|^2+\|\nabla u\|^2\right).\non
 \eea
 Summing up, we get
\be J'_1 \leq
-\frac{7k\varepsilon\gamma}{8}\Big\|\nabla\Delta\dfrac{\delta
E}{\delta\phi}\Big\|^2 +\left(1+C_1\right)\|\nabla \Delta
u\|^2+C\left(\Big\|\dfrac{\delta
E}{\delta\phi}\Big\|^4+\Big\|\dfrac{\delta
E}{\delta\phi}\Big\|^2+\|\nabla u\|^2\right),\non
 \ee
where $C_1$ is a constant depending on $\|u_0\|$, $\|\phi_0\|_{H^3}$
and coefficients of the system due to estimates \eqref{unilow} and
\eqref{phih3}.

The remaining terms $J'_2, ...,J'_6$ can be estimated as for
$J_2,...,J_7$ in the proof of Lemma \ref{highorder1} by using a
similar argument with minor modifications (replacing $\dfrac{\delta
E}{\delta\phi}$ in $J_2,...,J_7$ by $\Delta\dfrac{\delta
E}{\delta\phi}$). Consequently,
 \bea \sum_{i=2}^6 J'_i &\leq&
\frac{k\varepsilon\gamma}{4}\Big\|\nabla\Delta\dfrac{\delta
E}{\delta\phi}\Big\|^2+C\|\nabla u\|^2+C\Big\|\Delta\dfrac{\delta
E}{\delta\phi}\Big\|^2\non\\
 &\leq& \frac{3k\varepsilon\gamma}{8}\Big\|\nabla\Delta\dfrac{\delta
E}{\delta\phi}\Big\|^2+C\|\nabla u\|^2+C\Big\|\dfrac{\delta
E}{\delta\phi}\Big\|^2.
 \non
 \eea
Set
 $$\eta_1= \dfrac{\mu}{4(1+C_1)}.$$
Now we turn to estimate $I_2$:
 \bea I_2 &\leq& \frac{\mu}{8}\|\nabla\Delta
u\|^2+\|\nabla\phi\|_{\mathbf{L}^\infty}^2\Big\|\nabla\dfrac{\delta
E}{\delta\phi}\Big\|^2+\|\Delta\phi\|_{L^6}^2\Big\|\dfrac{\delta
E}{\delta\phi}\Big\|^2_{L^3}  \non\\
&\leq& \frac{\mu}{8}\|\nabla\Delta u\|^2+C\left(\Big\|\dfrac{\delta
E}{\delta\phi}\Big\|^{\frac43}\Big\|\nabla\Delta\dfrac{\delta
E}{\delta\phi}\Big\|^{\frac23}+\Big\|\dfrac{\delta
E}{\delta\phi}\Big\|^\frac53\Big\|\nabla\Delta\dfrac{\delta
E}{\delta\phi}\Big\|^\frac13+ \Big\|\dfrac{\delta
E}{\delta\phi}\Big\|^2 \right).
\non\\
&\leq& \frac{\mu}{8}\|\nabla\Delta
u\|^2+\frac{k\varepsilon\gamma\eta_1}{4}\Big\|\nabla\Delta\dfrac{\delta
E}{\delta\phi}\Big\|^2+C\Big\|\dfrac{\delta
E}{\delta\phi}\Big\|^2.\non
 \eea
Collecting the above estimates together, we deduce that \bea &&
\frac{d}{dt}\left(\|\Delta u\|^2+ \eta_1\Big\|\nabla\dfrac{\delta
E}{\delta\phi}\Big\|^2\right)+\frac{k\varepsilon\gamma\eta_1}{2}\Big\|\nabla\Delta\dfrac{\delta
E}{\delta\phi}\Big\|^2 + \mu\|\nabla \Delta
u\|^2\non\\
&\leq& C\left(\Big\|\dfrac{\delta
E}{\delta\phi}\Big\|^4+\Big\|\dfrac{\delta
E}{\delta\phi}\Big\|^2+\|\nabla u\|^{10}+\|\nabla u\|^2\right)\non\\
&\leq&
 C\left(\mathcal{A}^5(t)+\mathcal{A}(t)\right).\label{uh2phih5}
 \eea
For the sake of simplicity, we denote
$$ \mathcal{Q}(t) =\frac{1+\|\nabla
u(t)\|_{\mathbf{L}^p}^{\frac{2p}{2p-3}} }{1+\ln(e+\|\nabla
u(t)\|_{\mathbf{L}^p})}.$$ For $\frac32< p\leq 6$ we infer from
\eqref{rehighorder} that \bea \dfrac{d }{dt}\mathcal{A}(t) &\leq&
C\left(1+\|\nabla
u\|_{\mathbf{L}^p}^{\frac{2p}{2p-3}}\right)\mathcal{A}(t)\non\\
&\leq& C_* \mathcal{Q}(t)\left[1+\ln\big(e+\|\Delta
u(t)\|\big)\right]\mathcal{A}(t),\label{cstar}
 \eea
  where $C_*$ is a constant depending on $\|u_0\|$, $\|\phi_0\|_{H^3}$
and coefficients of the system.

Because of \eqref{integrability condition1}, we denote $\int_{0}^{T}
\mathcal{Q}(t)
 dt=M<+\infty$. Fix $\epsilon\in \Big(0,  \dfrac{1}{5C_*}
\Big]$. Then there exist
$0=\tilde{t}_0<\tilde{t}_1<...<\tilde{t}_{N-1}<\tilde{t}_{N}=T$ such
that
 \[\int_{\tilde{t}_{i-1}}^{\tilde{t}_i}  \mathcal{Q}(t)
 dt\leq \frac{\epsilon}{2}, \quad \forall \ i=1,2,...,N,\ \ \text{with} \ N=\left[\frac{2M}{\epsilon}\right]+1.
 \]
 Set $t_0=\tilde{t}_0=0$ and $t_{N+1}=\tilde{t}_{N}=T$. It follows from our assumption on the initial data that $\mathcal{A}(0)<+\infty$.
 Due to \eqref{int}, for each $i=1,2,...,N$, there exists $t_i\in (\tilde{t}_{i-1}, \tilde{t}_{i})$ such that $\mathcal{A}(t_i)<+\infty$. Moreover,
\bea && \int_{t_{i}}^{t_{i+1}} \mathcal{Q}(t)
 dt\leq \int_{\tilde{t}_{i}}^{\tilde{t}_{i+1}} \mathcal{Q}(t)
 dt\leq \frac{\epsilon}{2}\leq \epsilon,\quad i=0,N,\label{ai1}\\
&& \int_{t_{i}}^{t_{i+1}} \mathcal{Q}(t)
 dt\leq \int_{\tilde{t}_{i-1}}^{\tilde{t}_i} \mathcal{Q}(t)
 dt+\int_{\tilde{t}_{i}}^{\tilde{t}_{i+1}} \mathcal{Q}(t)
 dt\leq \epsilon, \quad \text{for} \ i=1,2,...,N-1.\label{ai2}
\eea 
We can prove the required result by an iteration argument from
$i=0$ to $i=N$. For $i=0$, it follows from the Gronwall inequality
and \eqref{ai1} that
 \bea \mathcal{A}(t) &\leq&
\mathcal{A}(0)\exp\left(
C_*\left[1+\ln\big(e+\displaystyle\mbox{sup}_{[0, t]} \|\Delta
u(\cdot)\|\big)\right]\int_{0}^t  \mathcal{Q}(s)ds \right)\non\\
&\leq&
\mathcal{A}(0)e^{C_*\epsilon}\left(e+\displaystyle\mbox{sup}_{[0,
t]} \|\Delta u(\cdot)\|\right)^{C_*\epsilon},\quad \forall \,t \in
[0, t_1]. \label{kkk}
 \eea
We infer from \eqref{uh2phih5} and \eqref{kkk} that for $t \in [0,
t_1]$, it holds
  \bea
  &&\dfrac{d}{dt}\left(\|\Delta
u\|^2(t)+\eta_1\Big\|\nabla\dfrac{\delta E}{\delta\phi}\Big\|^2(t)
\right) \leq C\left(e+\displaystyle\mbox{sup}_{[0, t]} \|\Delta
u\|(\cdot)\right).
  \label{highest order energy law}
   \eea
Since $\|\Delta u(0)\|$ and $\Big\|\nabla\dfrac{\delta
E}{\delta\phi}\Big\|(0)$ are bounded due to our assumption on the
initial data, integrating \eqref{highest order energy law} from $0$
to $t$, we get
 \bea
 &&\|\Delta
u(t)\|^2+\eta_1\Big\|\nabla\dfrac{\delta E}{\delta\phi}\Big\|^2(t) \non\\
&\leq& \|\Delta u(0)\|^2+\eta_1\Big\|\nabla\dfrac{\delta
E}{\delta\phi}\Big\|^2(0)+CT\left(e+\displaystyle\mbox{sup}_{[0,
t]}\|\Delta u(\cdot)\|\right), \quad \forall\, t\in [0,t_1]. \non
 \eea
Then taking the supremum of both sides for $t\in [0,t_1]$,  we can
see that
 \bea
 &&\displaystyle\mbox{sup}_{[0, t_1]}\left(\|\Delta
u(\cdot)\|^2+\eta_1\Big\|\nabla\dfrac{\delta
E}{\delta\phi}(\cdot)\Big\|^2\right)
 \non\\
 &\leq& \|\Delta u(0)\|^2+\eta_1\Big\|\nabla\dfrac{\delta
E}{\delta\phi}\Big\|^2(0)
 +CT\left(e+\displaystyle\mbox{sup}_{[0,
t_1]}\|\Delta u(\cdot)\|\right)
 \non\\
&\leq& \|\Delta u(0)\|^2+\eta_1\Big\|\nabla\dfrac{\delta
E}{\delta\phi}\Big\|^2(0)+\frac12\displaystyle\mbox{sup}_{[0,
t_1]}\|\Delta u(\cdot)\|^2+C_{T},\label{kk}
 \eea
which indicates that $\|\Delta u\|$ and $\Big\|\nabla\dfrac{\delta
E}{\delta\phi}\Big\|$ are uniformly bounded on $[0, t_1]$.

Then we can repeat the above argument for $i=1,...,N$ such that on
each interval $[t_{i}, t_{i+1}]$, it holds
\be \displaystyle\mbox{sup}_{[t_{i}, t_{i+1}]}\left(\|\Delta
u(\cdot)\|^2+\eta_1\Big\|\nabla\dfrac{\delta
E}{\delta\phi}(\cdot)\Big\|^2\right)
 \leq \|\Delta u(t_{i})\|^2+\eta_1\Big\|\nabla\dfrac{\delta
E}{\delta\phi}\Big\|^2(t_{i})+C_{T},\label{kki}
 \ee
 where the bound of $\|\Delta u(t_{i})\|$, $\Big\|\nabla\dfrac{\delta
E}{\delta\phi}\Big\|(t_{i})$ are given by the estimates in the
previous step on $[t_{i-1}, t_{i}]$. Consequently, we get
 \be \|u\|_{L^\infty(0, T; \mathbf{H}^2_p)}\leq C, \ \ \ \|\phi\|_{L^\infty(0, T;
H^5_p)}\leq C,\non
 \ee
which indicate that $[0, T)$ cannot be the maximal interval of
existence, and the solution $(u, \phi)$ can be extended beyond $T$.

 \textbf{Case} (ii). We re-estimate the terms $I_1$ and $J'_{1f}$ in a different way by  using the uniform estimates \eqref{unilow} and
 \eqref{phih3}. The estimate for $I_1$ can be done as follows:
 \bea I_1&\leq& C\|\nabla \Delta u\|(\|\nabla u\|_{\mathbf{L}^4}^2+\|\Delta u\|_{\mathbf{L}^\frac{2p}{p-2}}\|u\|_{\mathbf{L}^p}) \non\\
&\leq& C\|\nabla \Delta u\|\left(\|\Delta u\|_{\mathbf{L}^\frac{2p}{p-2}}\|u\|_{\mathbf{L}^p}+\|u\|_{\mathbf{L}^p}^2\right) \non\\
&\leq& C\|\nabla \Delta u\|\|\nabla \Delta u\|^\frac3p\|\Delta u\|^{1-\frac3p}\|u\|_{\mathbf{L}^p}\non\\
&\leq& \frac{\mu}{8}\|\nabla\Delta
u\|^2+C\|u\|_{\mathbf{L}^p}^\frac{2p}{p-3}\|\Delta u\|^2. \non
 \eea
Meanwhile, using Lemma \ref{eqes}, we get
 \bea
  J'_{1f}
  &\leq&
C\|u\|_{\mathbf{L}^\infty}^2\|\phi\|_{H^4}^2 \leq
C\|u\|_{\mathbf{L}^\infty}^2\|\phi\|_{H^5}\|\phi\|_{H^3}  \non\\
  &\leq& C(\|
u\|\|\nabla\Delta u\|+\|u\|^2)\left(\Big\|\nabla \dfrac{\delta
E}{\delta\phi}\Big\|+C\right) \non\\
 &\leq& \frac{\mu}{8}\|\nabla\Delta u\|^2+C\left(\Big\|\nabla \dfrac{\delta
E}{\delta\phi}\Big\|^2+\|\nabla u\|^2\right).\non \eea
 The other terms in $I_2$, $J'_1,...,J'_6$ are estimated in the same way as in Case (i). Then we deduce that
 \bea
&& \frac{d}{dt}\left(\|\Delta u\|^2+ \eta_1\Big\|\nabla\dfrac{\delta
E}{\delta\phi}\Big\|^2\right)+\frac{k\varepsilon\gamma\eta_1}{2}\Big\|\nabla\Delta\dfrac{\delta
E}{\delta\phi}\Big\|^2 + \mu\|\nabla \Delta
u\|^2\non\\
&\leq& C\|u\|_{\mathbf{L}^p}^\frac{2p}{p-3}\|\Delta u\|^2+
C\left(\Big\|\nabla \dfrac{\delta E}{\delta\phi}\Big\|^2+\Big\|
\dfrac{\delta
E}{\delta\phi}\Big\|^2+\|\nabla u\|^2\right)\non\\
&\leq& C\left(1+\|u\|_{\mathbf{L}^p}^\frac{2p}{p-3}
\right)\left(e+\|\Delta u\|^2+ \eta_1\Big\|\nabla\dfrac{\delta
E}{\delta\phi}\Big\|^2\right)\non\\
&\leq&
C\frac{1+\|u\|_{\mathbf{L}^p}^\frac{2p}{p-3}}{1+\ln(e+\|u\|_{\mathbf{L}^\infty})}\left[
1+\ln(e+\|\Delta u\|^2)\right]   \left(e+\|\Delta u\|^2+
\eta_1\Big\|\nabla\dfrac{\delta
E}{\delta\phi}\Big\|^2\right),\label{uh2phih5b}
 \eea
 where we have used the Poincar\'e inequality
 $$ \Big\| \dfrac{\delta
E}{\delta\phi}\Big\|^2\leq C\Big\|\nabla\dfrac{\delta
E}{\delta\phi}\Big\|^2+C\Big|\int_Q \dfrac{\delta E}{\delta\phi}
dx\Big|^2\leq C\Big\|\nabla\dfrac{\delta
E}{\delta\phi}\Big\|^2+C(\|\phi\|_{H^2}).$$
 Then we infer from
\eqref{uh2phih5b} and the Gronwall inequality that for all $t\in
[0,T]$,
 \bea
 && \ln\left(1+\ln\Big(e+\|\Delta u(t)\|^2+ \eta_1\Big\|\nabla\dfrac{\delta
E}{\delta\phi}\Big\|^2(t)\Big)\right)\non\\
&\leq& C  \ln\left(1+\ln\Big(e+\|\Delta u_0\|^2+
\eta_1\Big\|\nabla\dfrac{\delta
E}{\delta\phi}\Big\|^2(0)\Big)\right)
\int_0^T\frac{1+\|u(t)\|_{\mathbf{L}^p}^\frac{2p}{p-3}}{1+\ln(e+\|u(t)\|_{\mathbf{L}^\infty})}dt,\non
 \eea
 which together with \eqref{integrability condition2} implies that
 \be \|u\|_{L^\infty(0, T; \mathbf{H}^2_p)}\leq C, \ \ \ \|\phi\|_{L^\infty(0, T;
H^5_p)}\leq C.\non
 \ee
 The proof is complete.
\end{proof}

\section{Stability}\setcounter{equation}{0}

Denote the total energy of the system \eqref{NS}--\eqref{IC}  by
  \[\mathcal{E}(t)=\frac12\|u(t)\|^2+E(\phi(t)).\]
We recall that $\mathcal{E}(t)$ satisfies the basic energy law
\eqref{basic energy law}, which characterizes the dissipative nature
of the problem \eqref{NS}--\eqref{IC}. Inspired by \cite{LL95} for
the liquid crystal system, we can show that if the initial datum is
regular and the total energy $\mathcal{E}(t)$ cannot drop too much
for all time, then our problem \eqref{NS}--\eqref{IC} admits a
unique bounded global strong solution.

\begin{proposition}\label{small data proposition}
Let $n=3$. For any initial data $(u_0, \phi_0)\in
\dot{\mathbf{V}}\times H^4_p$, there exists a constant
$\varepsilon_0 \in(0,1)$, depending on
$\|u_0\|_{\mathbf{H}^1}$,$\|\phi_0\|_{H^4}$ and coefficients of the
system such that either

(i) The problem \eqref{NS}--\eqref{IC} has a unique global strong
solution $(u,\phi)$ with uniform-in-time estimate
 \begin{equation}
 \|u(t)\|_{\mathbf{V}}+\|\phi(t)\|_{\mathbf{H}^4}\leq C,\quad \forall\, t \geq 0,\label{unihigh}
 \end{equation}
or (ii) there is a $T_*\in (0,+\infty)$ such that $\mathcal{E}(T_*)
\leq \mathcal{E}(0)-\varepsilon_0$.
\end{proposition}

\begin{proof}
The proof is based on the higher-order differential inequality
\eqref{higha} (cf. Lemma \ref{highorder1}) and an argument similar
to that in \cite{LL95, SW11}. We only sketch it here for readers'
convenience. For any initial data $(u_0, \phi_0)\in
\dot{\mathbf{V}}\times H^4_p$,  let $L$ be a constant such that
$\|\nabla u_0\|^2+\Big\|\dfrac{\delta E}{\delta\phi}(0)\Big\|^2\leq
 L$. It follows from Lemma \ref{eqes} that $\|\phi_0\|_{H^4}$ can be bounded in terms of  $L$ and $\|\phi_0\|_{H^2}$. Then by Propositions \ref{low} and \eqref{comp}, $\|u(t)\|$ and $\|\phi(t)\|_{H^3}$ can be bounded by a constant depending  $L$, $\|\phi_0\|_{H^2}$ and coefficients of the system.
 Then we can fix the constant $\eta$ in the definition of
 $\mathcal{A}(t)$ (cf. \eqref{eta}) and $C_*$ in \eqref{higha}. Consider the ODE problem
 \be
 \frac{d}{dt}Y(t)=C_*[(Y(t))^3+Y(t)],\quad
 Y(0)=\max\{1,\eta\}L\geq \mathcal{A}(0)\label{ODE}
 \ee  The maximal existence time $T_{max}$ of the unique local solution $Y(t)$ is determined by $Y(0)$ and $C_*$. Now we take
  $$t_0=\frac12 T_{max}(Y(0), C_*),\quad \varepsilon_0=\frac{L t_0}{2}\min\left\{\mu,
  \gamma\right\}.$$
  If (ii) is not true, we have $ \mathcal{E}(t) \geq \mathcal{E}(0)-\varepsilon_0$ for all $t\geq0$.
From the basic energy law \eqref{basic energy law}, we infer that
 \be
\int_{\frac{t_0}{2}}^{t_0}\mathcal{A}(t)dt\leq
 \int_0^{\infty}\mathcal{A}(t)dt \leq \kappa\varepsilon_0,\quad
 \text{with}\ \kappa=\max\{1,\eta\}\max\{\mu^{-1}, \gamma^{-1}\}.\non
 \ee
 Hence, there exists a $t_*
\in [\frac{t_0}{2}, t_0]$ such that $\mathcal{A}(t^*)\leq
\frac{2\kappa\varepsilon_0}{t_0}\leq Y(0)$. Take $t_*$ as the
initial time for \eqref{ODE} with $Y(t^*)=Y(0)$, we infer from the
above argument that $Y(t)$ and thus $\mathcal{A}(t)$ remain bounded
at least on $[0, \frac{3t_0}{2}]\subset [0,t_*+t_0]$ with the same
bound as that on $[0,t_0]$ (since the bound for $Y(t)$ is the same).
An iteration argument shows that $\mathcal{A}(t)$ is bounded for all
$t\geq 0$. The proof is complete.
\end{proof}
\begin{corollary}[Eventual regularity of weak solutions in $3D$]\label{evereg}
When $n=3$, let $(u, \phi)$ be a global weak solution of the problem
\eqref{NS}--\eqref{IC}. Then there exists a time $T_0 \in (0,
+\infty)$ such that $(u, \phi)$ becomes a strong solution on $[T_0,
+\infty)$.
\end{corollary}
\begin{proof}
It follows from \eqref{unilow}, \eqref{int} and Lemma \ref{eqes}
that
 there exists a time $T_1>0$ such that $\|u(T_1)\|$, $\|\phi(T_1)\|_{H^4}$ are bounded.
Taking $T_1$ as the initial time, we can fix $L$ in Proposition
\ref{small data proposition} and thus $\varepsilon_0$. \eqref{int}
yields that there exists a $T_0>T_1$ such that
$$\|\nabla u(T_0)\|^2+\Big\|\dfrac{\delta
E}{\delta\phi}(T_0)\Big\|^2\leq
 L, \quad \int_{T_0}^{+\infty}\big(\mu\|\nabla u\|^2+\gamma\Big\|\dfrac{\delta
E}{\delta\phi} \Big\|^2\big)dt
 \leq \varepsilon_0.$$
 Taking $T_0$ as the initial time, we can apply the argument for Proposition \ref{small data proposition} that $(u,\phi)$ will be bounded in $\mathbf{H}^1\times H^4$
 after $T_0$.
\end{proof}
 \begin{definition}
 We say $\phi^*\in H^2_p$ is a local minimizer of the elastic energy
 $E(\phi)$, if there exists a $\delta>0$,  $E(\phi^*) \leq E (\phi)$ for all $\phi\in H^2_p$ satisfying $\|\phi-\phi^*\|_{H^2}<\delta$. If for all $\phi\in H^2_p$, $E(\phi^*) \leq E (\phi)$, then $\phi^*$ is an absolute minimizer.
 \end{definition}
\begin{lemma}
  Let $\mathcal{B}$ be a bounded closed convex subset of $H^2_p$. The approximate elastic energy  $E(\phi)$
   admits at least one minimizer $\phi^*\in \mathcal{B}$ such that $E(\phi^*)= \displaystyle{\inf_{\phi\in \mathcal{B}}}E(\phi)$.
  \end{lemma}
 \begin{proof}
 Since $E(\phi)\geq 0$ for all $\phi\in \mathcal{B}$ and $\displaystyle{\lim_{\phi\in \mathcal{B},\ \|\phi\|_{H^2}\to+\infty}} E(\phi)= +\infty$, $E(\phi)$ has a bounded minimizing sequence $\phi_n\in \mathcal{B}$ such that
     \be
     E(\phi_n) \rightarrow \displaystyle{\inf_{\phi\in \mathcal{B}}}E(\phi). \label{2.5}
     \ee
Recalling the definition of $E(\phi)$ \eqref{modified energy}, we
can rewrite $E$ in the following form:
    \be
    E(\phi)=\frac{k\varepsilon}{2}\|\Delta \phi \|^2+F(\phi)\non
    \ee
    with
    \bea
    F(\phi)&=&\frac{k}{\varepsilon}\int_Q \nabla \phi\cdot \nabla (\phi^3-\phi) dx+ \frac{k}{2\varepsilon^3}\int_Q (\phi^3-\phi)^2 dx\non\\
    &&\ \ +\frac12M_1(A(\phi)-\alpha)^2+\frac12M_2(B(\phi)-\beta)^2.\non
    \eea
    Since  $\phi_n$ is bounded in $H^2$, there is a subsequence, still denoted by $\phi_n$,
       such that $\phi_n$ weakly converges to a certain function $\phi^*$ in $H^2$. We infer from the compact Sobolev embedding theorem ($n=3$) that
    $\phi_n$ strongly converges to $\phi^*$ in $L^\infty$ and $H^1$. It turns out that $F(\phi_n) \rightarrow
    F(\phi^*)$. Since $\|\Delta \phi\|^2$ is weakly lower semi-continuous, it follows from
\eqref{2.5} that $E(\phi^*)=\displaystyle{\inf_{\phi\in
\mathcal{B}}}E(\phi)$. Using the elliptic estimate and a bootstrap
argument, we see that the minimizer $\phi^*$ is in fact smooth. The proof is complete.
 \end{proof}
 \begin{remark}
 If $\phi$ is a minimizer of $E(\phi)$, then it is a
critical point of $E(\phi)$. It is easy to verify that any critical
point of $E(\phi)$ in $H^2_p$
 is equivalent to
   a weak solution to the forth-order nonlocal elliptic problem
 \be
 \dfrac{\delta E}{\delta\phi}=0, \quad \text{with} \ \phi(x+e_i)=\phi(x).
 \label{se}
  \ee
\end{remark}
In order to prove our stability result, we introduce the following \L
ojasiewicz--Simon type inequality whose proof is postponed to Section 6.2.
 \bl
   \label{ls}
   Suppose $n=3$. Let $\psi$ be a critical point of the elastic energy $E$.
   There exist constants $\beta>0$, $\theta\in(0, \frac12)$
   depending on $\psi$ such that for any $\phi\in H^4_p(Q)$
   with $\|\phi-\psi\|_{H^2}<\beta$, it holds
   \be
   \Big\|\dfrac{\delta
E}{\delta\phi}\Big\| \geq  |E(\phi)-E(\psi)|^{1-\theta}.
   \label{LoSi}
   \ee
\el
Now we state the main result of this section.
\begin{theorem}\label{theorem on near equilibrium case}
 Let $\phi^\ast\in H_p^4(Q)$ be a local minimizer of $E(\phi)$. For any $R>0$, consider the initial data
 $$
 (u_0, \phi_0)\in \mathbf{B}=\{(u, \phi) \in \dot{\mathbf{V}}\times H^4_p(Q): \  \|u\|_{\mathbf{H}^1}\leq R, \ \|\phi_0-\phi^*\|_{H^4}\leq R\}.
 $$
 For any $\epsilon>0$, there exists a constant $\sigma\in (0,\delta)$ that may depend on
$\phi^\ast$, $R$, $\epsilon$ and coefficients of the system such
that if the initial data $(u_0, \phi_0)\in \mathbf{B}$ satisfies the
condition \be \|u_0\|+\|\phi_0-\phi^\ast\|_{H^2}\leq\sigma,
\label{small} \ee then the problem
 \eqref{NS}--\eqref{IC} admits a unique global strong solution satisfying
 \be
 \|\phi(t)-\phi^\ast\|_{H^2}\leq\epsilon,\quad \forall \,t\geq0. \label{Lya}
 \ee
 \end{theorem}

\begin{proof}
 If $\|u_0\|_{\mathbf{H}^1}\leq R$ and $\|\phi_0-\phi^*\|_{H^4}\leq R$, then the constant $\varepsilon_0$ in
  Proposition \ref{small data proposition} depends on $\phi^*$, $R$ and coefficients of the system.
  It follows from Proposition \ref{low} and Proposition \ref{comp} that $\|u(t)\|$ and
 $\|\phi(t)\|_{H^3}$ are uniformly bounded (by a constant depending on $\phi^*$, $R$ and coefficients of the system).
  In what follows, we denote by $C$, $C_i$ generic constants that only depend on $R$, $\phi^*$
 and coefficients of the system.

By a direction computation, we get
 \bea
 E(\phi_0)-E(\phi(t))&=& [E_\varepsilon(\phi_0)-E_\varepsilon(\phi_0)]+\frac12M_1[(A(\phi_0)-\alpha)^2-(A(\phi)-\alpha)^2] \non\\
 &&+\frac12M_2[(B(\phi_0)-\beta)^2-(B(\phi)-\beta)^2]\non\\
 &:=& F_1+F_2+F_3,\non
 \eea
where 
\bea F_1&\leq&
C\big(\|\Delta\phi_0-\Delta\phi\|+\|(\phi_0-\phi)(\phi_0^2+\phi^{2}+\phi_0\phi+1) \| \big)  \non\\
&\leq& C\big(\|\Delta\phi_0-\Delta\phi \|+\|\phi_0-\phi
\|\|\phi_0^2+\phi^{2}+\phi_0\phi+1\|_{L^\infty} \big)
\non\\
&\leq& C\|\phi_0-\phi\|_{H^2}, \label{F1}
 \eea
 \be
  F_2 \leq C\|\phi_0-\phi\|_{L^1}(\|\phi_0+\phi\|_{L^1}+2|\alpha| \Big) \leq C\|\phi_0-\phi\|,  \label{F2} 
  \ee
  \bea
F_3&\leq&C\|\nabla\phi_0+\nabla\phi\|\|\nabla\phi_0-\nabla\phi\|
+C\|\phi_0+\phi\|\|\phi_0^2+\phi^{2}-2\|_{L^\infty}\|\phi_0-\phi\|   \non\\
&\leq& C\|\phi_0-\phi\|_{H^1}. \label{F3}
 \eea
 Since the total energy $\mathcal{E}$ is decreasing in time, we infer from the above estimate that
 \bea
 0\leq
 \mathcal{E}(0)-\mathcal{E}(t)&=&\frac12\|u_0\|^2-\frac12\|u(t)\|^2+E(\phi_0)-E(\phi(t))
 \non\\
 &\leq& \frac12\|u_0\|^2+E(\phi_0)-E(\phi(t))\non\\
 &\leq& \frac12\|u_0\|^2+C_1\|\phi(t)-\phi_0\|_{H^2}.\label{Ediff}
 \eea
 Let $\beta$ denote the constant in Lemma \ref{ls} that depends only on $\psi=\phi^*$.
 Denote
 $$\varpi=\min\left\{1, \varepsilon_0^\frac12, \epsilon, \frac{\delta}{2}, \frac{\beta}{2}, \frac{3\varepsilon_0}{4C_1}\right\}.$$

 We assume that $\sigma\leq \frac14\varpi$. Let $\tilde{T}$ be the smallest finite time for which
 $\|\phi(\tilde{T})-\phi^*\|_{H^2}\geq \varpi$. Then by the proof of Proposition \ref{small data proposition},
 the problem admits a bounded strong solution on $[0,\tilde{T})$.
 If there exists $t_{*}\in (0,\tilde{T})$ such that
 $\mathcal{E}(t_{*})=E(\phi^*)$, since $\phi^*$ is the local
 minimizer and $\|\phi(t_{*})-\phi^*\|_{H^2}<\varpi<\delta$,
 we deduce from \eqref{basic energy law} that $\|\nabla u(t)\|=
\Big\|\dfrac{\delta E}{\delta\phi}(t)\Big\|= 0$ for $t\geq t_*$. It
follows from \bea
 \|\phi_t\|&\leq& \|u\cdot\nabla \phi\|+\gamma \Big\|\dfrac{\delta E}{\delta\phi}\Big\|\leq C\|\nabla u\|\|\nabla \phi\|_{\mathbf{L}^3}+\gamma \Big\|\dfrac{\delta E}{\delta\phi}\Big\|\non\\
 &\leq& C\left(\|\nabla u\|+\Big\|\dfrac{\delta E}{\delta\phi}\Big\|\right)\label{dt}
 \eea
that for $t\geq t_*$, $\|\phi_t\| =0$. Namely, $\phi$ is independent
of time for $t\geq t_*$. As a result, $u(t_{*})=0$ and
$\phi(t_{*})=\phi^{**}$, where
 $\phi^{**}$ is also a local minimizer (but possibly different from
 $\phi^*$). Due to the uniqueness of strong solution, the evolution starting from $t_{*}$ will be
 stationary. The proof is complete in this case.

 We proceed to work with the case that $\mathcal{E}(t)>E(\phi^*)$ for all $t\in [0,\tilde{T})$.
 From the definition of $\tilde{T}$, we see that the conditions in Lemma \ref{ls} are satisfied with $\psi=\phi^*$, on the interval
 $[0,\tilde{T})$. Consequently,
 \bea
 -\frac{d}{dt}(\mathcal{E}(t)-E(\phi^*))^\theta \geq
 \frac{\mu\|\nabla u\|^2+\gamma\Big\|\dfrac{\delta E}{\delta\phi}\Big\|^2}{\frac12\|u\|^{2(1-\theta)}+ \Big\|\dfrac{\delta E}{\delta\phi}\Big\|}
  \geq C\left(\|\nabla u\|+\Big\|\dfrac{\delta E}{\delta\phi}\Big\|\right), \quad \forall\  t\in [0,\tilde{T}).\label{intb}
 \eea
 We infer from \eqref{dt} that
 \bea
 \int_0^{\tilde{T}}\|\phi_t(t)\|dt &\leq&
 C(\mathcal{E}(0)-E(\phi^*))^\theta
 \leq C(\|u_0\|^{2\theta}+
 |E(\phi_0)-E(\phi^*)|^\theta)\non\\
 &\leq& C \left(\|u_0\|^{2\theta}+\|\phi_0-\phi^*\|_{H^2}^\theta\right),\non
 \eea
 which implies that
 \bea
 && \|\phi(\tilde{T})-\phi^*\|_{H^2}\non\\
 &\leq&
 \|\phi(\tilde{T})-\phi_0\|_{H^2}+\|\phi_0-\phi^*\|_{H^2}\non\\
 &\leq&
 C\|\phi(\tilde{T})-\phi_0\|_{H^3}^\frac23\|\phi(\tilde{T})-\phi_0\|^\frac13+\|\phi_0-\phi^*\|_{H^2}\non\\
 &\leq& C\left( \int_0^{\tilde{T}}\|\phi_t(t)\|dt
 \right)^\frac13+\|\phi_0-\phi^*\|_{H^2}\non\\
 &\leq& C_2 \left(\|u_0\|^{\frac{2\theta}{3}}+
 \|\phi_0-\phi^*\|_{H^2}^\frac{\theta}{3}\right)+\|\phi_0-\phi^*\|_{H^2}.\label{difff}
 \eea
Taking
 \be
 \sigma\leq\min\left\{\frac{\varpi}{4},
 \Big(\frac{\varpi}{4C_2}\Big)^\frac{3}{\theta}\right\},\label{ssig}
 \ee
 we easily deduce from \eqref{difff} that
 $\|\phi(\tilde{T})-\phi^*\|_{H^2}\leq \frac34\varpi<\varpi$, which leads to
 a contradiction with the definition of $\tilde{T}$. Hence, $\tilde{T}=+\infty$ and
 there holds
 \be
 \|\phi(t)-\phi^*\|_{H^2}\leq \varpi\leq \epsilon,\quad \forall \ t\geq 0.\label{close1}
 \ee
 Therefore,
 \be
 \|\phi(t)-\phi_0\|_{H^2}\leq
 \|\phi(t)-\phi^*\|_{H^2}+\|\phi^*-\phi_0\|_{H^2}\leq \varpi+
 \sigma\leq \frac54\varpi.
 \ee
 then it follows from \eqref{Ediff} that
 \be
 \mathcal{E}(t)\geq \mathcal{E}(0)-\varepsilon_0,\quad \forall \, t\geq
 0.\label{cod}
 \ee
 By Proposition \ref{small data proposition}, we see that \eqref{NS}--\eqref{IC} admits a unique global strong solution that satisfies \eqref{Lya}. The proof is complete.
 \end{proof}

 \begin{corollary}
 Assume that the assumptions of Theorem \ref{theorem on near equilibrium case} are satisfied. The global strong solution $(u, \phi)$ has the
 following property:
 \be \lim_{t\rightarrow +\infty}
 (\|u(t)\|_{\mathbf{H}^1}+\|\phi(t)-\phi_\infty\|_{H^4})=0,\label{cgce}
 \ee
 where $\phi_\infty\in H^4_p$ is a solution to \eqref{se} such that $E(\phi^*)=E(\phi_\infty)$.
 Moreover, there exists a positive constant $C$ depending on
 $u_0, \phi_0$ and coefficients of the system such that
 \be
 \|u(t)\|_{\mathbf{H}^1}+\|\phi(t)-\phi_\infty\|_{H^4}\leq C(1+t)^{-\frac{\theta'}{(1-2\theta')}}, \quad \forall\, t \geq
 0.\label{rate}
 \ee
$\theta' \in (0,\frac12)$ is the \L ojasiewicz exponent in Lemma
\ref{ls} depending on $\phi_\infty$.
 \end{corollary}
 \begin{proof}
 We infer from the higher-order energy inequality \eqref{higha} and uniform estimates \eqref{unihigh} that $\frac{d}{dt}\mathcal{A}(t)$ is bounded for $t>0$. On the other hand, the basic energy law implies that $\mathcal{A}(t)\in L^1(0,+\infty)$, then we have $\lim_{t\to +\infty}\mathcal{A}(t)=0$. Thus, we obtain the decay property of the velocity field $u$ in $\mathbf{V}$ and
 \be
 \lim_{t\to +\infty}\Big\|\dfrac{\delta E}{\delta \phi}(t)\Big\|=0.\label{decaypp}
 \ee
 Recalling the proof of Theorem \ref{theorem on near equilibrium case}, we have shown that
$\|\phi_t(t)\|\in L^1(0, +\infty)$. As a consequence, $\phi(t)$ will
converge in $L^2$ to a function $\phi_\infty\in H^4_p(Q)$ that
satisfies \eqref{se} due to \eqref{decaypp}. It follows from
\eqref{close1} that for sufficiently large $t$, we have
 \bea
  \|\phi_\infty-\phi^*\|_{H^2}&\leq&  \|\phi_\infty-\phi(t)\|_{H^2}+\|\phi(t)-\phi^*\|_{H^2} \non\\
  &\leq& C\|\phi_\infty-\phi(t)\|^\frac12_{H^4}\|\phi_\infty-\phi(t)\|^\frac12+\|\phi(t)-\phi^*\|_{H^2}\non\\
  &\leq& \min\{\delta, \beta\}.\non
 \eea
 Thus, applying Lemma \ref{ls} with $\phi=\phi_\infty$ and $\psi=\phi^*$, we have
 $$|E(\phi_\infty)-E(\phi^*)|^{1-\theta}\leq
 \Big\|\dfrac{\delta E}{\delta \phi}(\phi_\infty)\Big\|=0.$$
 The limit function $\phi_\infty$ is also a local minimizer of the energy $E$ and it will coincide with $\phi^*$ if the latter is isolated.
Finally, the proof for convergence rate \eqref{rate} is based on
Lemma \ref{ls} and higher-order differential inequalities as for the
liquid crystal systems \cite{SW11,WXL}. Since the proof is lengthy
but standard, we omit the details here.
 \end{proof}

\begin{remark}
We can also prove the long-time behavior for global weak solutions
with arbitrary initial data in contrast with smallness assumption
like \eqref{small}. Indeed, Corollary \ref{evereg} implies that any
global weak solution  $(u, \phi)$ to the problem
\eqref{NS}--\eqref{IC} will become a bounded strong one after a
sufficiently large time (eventual regularity). Then
 we can just make a shift of time and consider the long-time behavior of bounded strong solutions.
Applying Lemma \ref{ls}, we can use the Lojasiewicz--Simon technique
(cf., e.g., \cite{Chill}, see also \cite{Ab, GG09, GG10, LD, S83,
WXL} for applications to related models) to show that each weak
solution does converge to a single pair $(0,\phi_\infty)$ with
$\phi_\infty$ satisfying the stationary problem \eqref{se}. Besides,
one can obtain an estimate on convergence rate as \eqref{rate}.
\end{remark}

\section{Appendices}\setcounter{equation}{0}

\subsection{A formal physical derivation via energy variational approaches}
The energy variational approaches (in short, \textit{EnVarA}) provide unified
variational frameworks in studying complex fluids with
micro-structures (cf. e.g., \cite{ABML, HKL, YFLS}). From the energetic point of view, the system
\eqref{NS}--\eqref{phasefield} exhibits the competition between the
macroscopic kinetic energy and the microscopic membrane elastic
energy. The interaction or
coupling between different scales plays a crucial role in understanding complex fluids.
Based on the basic energy law \eqref{basic energy law}, we
shall perform a formal physical derivation of the induced elastic stress
through \textit{EnVarA}, which provides a further understanding of vesicle-fluid interactions.

The energetic variational treatment of complex fluids starts with
the energy dissipative law for the whole coupled system :
 \be
\dfrac{dE^{tot}}{dt}=-\mathcal{D},
 \non
 \ee
 where $E^{tot}=E^{kinetic}+E^{int}$ is the total
energy consisting of the kinetic energy and free energy and $\mathcal{D}$ is the dissipation function which is equal to  the
entropy production of the system in isothermal situations. The
\emph{EnVarA} combines the least action principle (for intrinsic and
short time dynamics) and the maximum dissipation principle (for long
time dynamics) into a force balance law that expands the
conservation law of momentum to include dissipations (cf. \cite{La96,
CH53, O53, O53-2}). The former gives us the Hamiltonian (reversible)
part of the system related to conservative forces, while the latter
provides the dissipative (irreversible) part related to dissipative
forces. In this way, we can distinguish the conservative and
dissipative parts among the induced stress terms.

The basic variable in continuum mechanics is the flow map $x(X, t)$,
(particle trajectory for any fixed $X$) . Here, $X$ is the original
labeling (the Lagrangian coordinate) of the particle, which is also
referred to as the material coordinate, while $x$ is the current
(Eulerian) coordinate and is also called the reference coordinate.
For a given velocity field $u(x, t)$, the flow map is defined by the
ordinary differential equations:
   \be
   x_t=u(x(X, t), t), \;\; x(X, 0)=X.\non
   \ee
The deformation tensor $\mathsf{F}$ associated with the flow field
is given by $\mathsf{F}_{ij}=\frac{\partial x_i}{\partial X_j}$. Note that we now have 
${\rm det}(\mathsf{F})=1$, which is equivalent to the incompressibility
condition $\nabla \cdot u=0$ (cf. e.g., \cite{LLZ05}). 
If there is no internal microscopic damping, the label function
$\phi$ satisfies the pure transport equation 
$$
\partial_t \phi+u\cdot\nabla\phi=0,
$$
 since a particle initially lying at $X$ will retain the initial label
$\phi_0(X)$ as time evolves. It is noted that the kinematic
assumption
$\dot{\phi}=\partial_t\phi +u\cdot\nabla\phi$ stands
for the influence of macroscopic dynamics on the microscopic scale.

The Legendre transform yields the action functional $\mathcal{A}$ of
the particle trajectories in terms of the flow map $x(X, t)$:
 \bea
\mathcal{A}(x)&=&\int_0^T\big(E^{kinetic}-E^{int} \big)dt
\non\\&=&\int_0^T\int_{\Omega_0}\Big[\frac12|x_t(X,t)|^2-E(\phi(x(X,t),t))\Big]{\rm det} (\mathsf{F})dXdt,
\eea 
 with $\Omega_0$ being the original domain. We take a one-parameter family of volume preserving
flow maps
\[ x^s(X, t) \ \ \mbox{with} \ \ x^0(X,t)=x(X,t), \ \ \frac{d}{ds}x^{s}(X,t)\Big|_{s=0}=v(X,t).  \]
The Eulerian velocity $\tilde{v}$ associated with $v$ is defined by $\tilde{v}(x(X,t),t)=v(X,t)$. We assume the volume-preserving condition, or
equivalently, ${\rm det}(\mathsf{F}^s)={\rm det}(\nabla_X x^s) =1$. In addition, given $\phi_0
: \Omega_0 \rightarrow \mathbb{R}$, we define the function
$\phi^s : \Omega^s \times [0, T]\rightarrow \mathbb{R}$ which
takes the constant value $\phi_0(X)$ along the particle trajectory
starting from $X$ under the motion $x^s$ such that
 \be
\phi^s(x^s(X, t), t)=\frac{\phi_0(X)}{{\rm det}(\mathsf{F}^s)}=\phi_0(X), \non
 \ee
 The functions $\phi$, $\phi^s$ are well defined because $x(\cdot,t)$, $x^s(\cdot,t)$ are one-to-one and onto for each $t$. Moreover, by definition, we have $\phi^0=\phi$. 
 Then by the chain rule, we get
 $$ \partial_{s}{\phi^s}(x^s)+\left(\frac{d}{ds}{x}^{s}\right)\cdot \nabla_{x^s}\phi^s(x^s)=\frac{d}{ds}\phi_0(X)=0.$$
 Taking $s=0$, we see that 
 \be
\partial_{s}{\phi^s}(x^s) \Big|_{s=0}+\tilde{v}\cdot\nabla_x\phi=0. \label{transport
equ for perturbation}
 \ee
Suppose that a motion $x$ is a critical point of the action $\mathcal{A}$, the least action principle yields that 
 \[
\delta_x\mathcal{A}=\frac{d}{ds}\mathcal{A}(x^s)\Big|_{s=0}=0,
\] 
for every smooth, one-parameter family of motions $x^s$ with $x^s(\cdot, 0)= x(\cdot,0)$
and $x^s(\cdot, T)=x(\cdot,T)$ (then we also have $v(\cdot, 0)=v(\cdot, T)=0$). After pushing forward to the Eulerian coordinates, we can see that 
  \bea
0&=&\int_0^T \left(\int_{\Omega_0}x_t\cdot
v_t dX-\int_\Omega \frac{\delta{E}}{\delta\phi}\partial_s\big(\phi^s(x^s)\big)\Big|_{s=0} dx\right) dt  \non\\
&=&\int_0^T \left(\int_{\Omega_0} x_t\cdot
v_t dX+\int_{\Omega} \frac{\delta{E}}{\delta\phi}\nabla_x\phi\cdot \tilde{v}\,dx\right)dt \non\\
&=&\int_0^T\left(\int_{\Omega_0} -x_{tt}\cdot
v dX+\int_\Omega \frac{\delta{E}}{\delta\phi}\nabla_x\phi\cdot \tilde{v}\,dx\right)dt \non\\
&=&-\int_0^T\int_{\Omega}\Big(\partial_t u+u\cdot\nabla{u}
-\frac{\delta{E}}{\delta\phi}\nabla_x\phi\Big)\cdot\tilde{v}\,dxdt. \non
\eea Since $\tilde{v}$ is an arbitrary divergence free vector field, we
formally derive the weak form of the Hamiltonian/conservative part
of the momentum equation
 \be
u_t+u\cdot\nabla{u}+\nabla{P}=\frac{\delta{E}}{\delta\phi}\nabla\phi,
\label{equ from LAP}
 \ee with $P$ serving as a Lagrangian multiplier
for the incompressibility condition.

On the other hand, the maximum dissipation principle concerns the dissipations of the system,
which represent the macroscopic long time dynamics. The diffusive interface method is imposed by an additional dissipation term
(relaxation) in the transport equation 
 \be
\phi_t+u\cdot\nabla\phi=-\gamma\frac{\delta{E}}{\delta\phi},
\label{full transport equ} 
 \ee 
 which indicates a gradient flow dynamics that is
another formulation of the near equilibrium, linear response theory
(cf. \cite{O53, O53-2}). 
We also want to include the dissipation in flow field caused by the flow viscosity $\mu$ such that the total dissipation is given by 
\[ \mathcal{D}=\mu\int_{\Omega}|\nabla{u}|^2dx+\gamma\int_{\Omega}\left|\frac{\delta{E}}{\delta\phi}\right|^2dx. \]
By equation \eqref{full transport equ}, we
can reformulate the dissipation functional as
\[ \mathcal{D}=\mu\int_{\Omega}|\nabla{u}|^2dx+\frac{1}{\gamma}\int_{\Omega}|\phi_t+u\cdot\nabla\phi|^2dx. \]
Then we perform a variation on the dissipation
functional with respect to the velocity $u$ in Eulerian coordinates
to get a weak form of the dissipative force balance law. Let
$u^{s}=u+s{v}$, where $v$ is an arbitrary smooth vector function
with $\nabla\cdot{v}=0$. 
 Then we have
  \bea 0&=&\delta_u\Big(\frac12\mathcal{D}
\Big)=\frac12\frac{d\mathcal{D}(u^{s})}{ds}\Big|_{s=0}\non\\
&=&\mu\int_{\Omega}(\nabla{u} :
\nabla{v})dx+\frac{1}{\gamma}\int_{\Omega}
(v\cdot\nabla\phi)(\phi_t+u\cdot\nabla\phi)dx \non\\
&=&\mu\int_{\Omega}(\nabla{u} :
\nabla{v})dx-\int_{\Omega}(v\cdot\nabla\phi)\frac{\delta{E}}{\delta\phi}dx
\non\\
&=&-\int_{\Omega}\Big(
\mu\Delta{u}+\frac{\delta{E}}{\delta\phi}\nabla\phi\Big)\cdot v dx,
 \non
  \eea 
  which yields the dissipative force balance law
 \be
-\nabla{P}+\mu\Delta{u}+\frac{\delta{E}}{\delta\phi}\nabla\phi=0,
\label{equ from MDP}
 \ee where $P$ serving as a Lagrangian multiplier
for the incompressibility condition $\nabla\cdot v=0$.

Our system \eqref{NS}--\eqref{phasefield} can be viewed as the hybrid of these two
conservative/dissipative systems \eqref{equ from LAP} and \eqref{equ from MDP} and taking into
account the total transport equation \eqref{full transport equ} for the phase function $\phi$. The above energy variational approach enables us to derive
the thermodynamic-consistent models involving different physics at
different scales.

\br We point out that the induced elastic force
$\frac{\delta{E}}{\delta\phi}\nabla\phi$ can be derived either from
least action principle or the maximum dissipation principle, which indicates that it can be recognized either as
conservative or dissipative. In contrast, $\mu\Delta{u}$ can only be
derived from the maximum dissipation principle, which is henceforth a dissipative term.

\er

\subsection{\L ojasiewicz--Simon type inequality }

We prove that a \L ojasiewicz--Simon type inequality holds in a
proper neighborhood of every critical point of the functional
$E(\phi)$.

First, We recall the definition of analyticity on Banach spaces
\cite[Definition 8.8]{Z}: Suppose~$X,\ Y$ are two Banach spaces. The
operator $T:\ D(T)\subseteq X\rightarrow Y$ is analytic if and only
if for any $x_0\in X$, there exists a small neighborhood of $x_0$
such that
 \be T(x_0+h)-T(x_0)=\sum_{n\geq
1}T_n(x_0)(h,...,h),\quad \forall\, h\in X, \ \ \|h\|_{X}<r<<1.\non
  \ee
Here $T_n(x_0)$ is a continuous symmetrical $n$-linear operator on
$X^n\rightarrow Y$ and satisfies
 \be \sum_{n\geq
1}\|T_n(x_0)\|_{L(X^n,Y)}\|h\|^n_X<+\infty.\non
 \ee
 \begin{proposition} \label{anal}
 Suppose $n=3$, we have
 \begin{itemize}
 \item[(1)] $E(\phi): H^4_p(Q)\to \mathbb{R}$ is analytic;
 \item[(2)] $\frac{\delta E}{\delta\phi}: H^4_p(Q)\to L^2_p(Q)$ is analytic;
 \item[(3)] for any $\psi\in H^4_p(Q)$ , $E''(\psi)$ is a Fredholm operator of index zero from $H^4_p(Q)$ to $L^2_p(Q)$.
 \end{itemize}
 \end{proposition}
 \begin{proof}
 (1) It follows from the definition of $E$ that it is the sum of integrations of polynomials in terms of $\Delta \phi$, $\nabla \phi$ and $\phi$. Since $\phi\in H^4_p(Q)$, then those functions belongs to $H^2_p(Q)$ which is a Banach algebra for the pointwise multiplication when $n=3$. Thus, $E(\phi): H^4_p(Q)\to \mathbb{R}$ is analytic.

 (2) Recalling \eqref{v1E}, we see that $\frac{\delta E}{\delta\phi}=k\varepsilon\Delta^2 \phi+H(\phi)$, where $H(\phi)$ is the sum of polynomials in terms of $\Delta \phi$, $\nabla \phi$ and $\phi$ that belong to $H^2_p(Q)$ when $\phi\in H^4_p(Q)$. Thus, our conclusion follows.

 (3)  For any $\psi, w_1, w_2\in H^4_p(Q)$, we calculate that
 \bea && E''(\psi)(w_1, w_2)\non\\
 &=&\frac{d}{ds} (E'(\psi+sw_1), w_2)|_{s=0}\non\\
 &=& \frac{d}{ds}\int_Q \left( kg(\psi+sw_1)+ M_1(A(\psi+sw_1)-\alpha)\right) w_2 dx\left. \right|_{s=0}\non\\
 && \ \ +
   \frac{d}{ds}\int_Q M_2(( B(\psi+sw_1)-\beta)f(\psi+sw_1)) w_2 dx \Big|_{s=0}\non\\
 &=&\int_Q k\varepsilon (\Delta^2 w_1)w_2 - \frac{6k}{\varepsilon} |\nabla \psi|^2 w_1 w_2- \frac{12k}{\varepsilon} (\psi\nabla \psi \cdot\nabla w_1) w_2 - \frac{12k}{\varepsilon}(\psi\Delta\psi)w_1w_2 dx
 \non\\
 && \ \ -\int_Q \frac{2k}{\varepsilon}(3\psi^2-1)(\Delta w_1) w_2 dx +\frac{k}{\varepsilon^2}\int_Q (15\psi^4-12\psi^2-1)w_1w_2 dx\non\\
 &&\ \  + M_1\int_Q w_1 dx \int_Q w_2 dx + M_2 \int_Q f(\psi) w_1 dx\int_Q f(\psi) w_2 dx
 \non\\
 &&\ \ 
 + M_2(B(\psi)-\beta)\int_Q\left[-\varepsilon \Delta w_1 + \frac{1}{\varepsilon}(3\psi^2-1)w_1\right] w_2 dx\non\\
 &:=& \int_Q k\varepsilon (\Delta^2 w_1)w_2 dx+ R(\psi)(w_1, w_2).\non
 \eea
 We observe that the leading order term of the linear operator $E''(\psi)$ is $k\varepsilon\Delta^2$. This forth-order operator (subject to periodic boundary conditions) from $H^4_p(Q)$ to $L^2_p(Q)$ can be  associated with the symmetric bilinear form $A: H^2_p(Q)\times H^2_p(Q)\to \mathbb{R}$ given by
 $$ A(f, g)=k\varepsilon\int_Q \Delta f\Delta g dx, \quad \forall\, f, g \in H^2_p(Q)$$
 such that by integration by parts $ \int_Q k\varepsilon\Delta^2 f gdx=A(f, g)$ for any $f, g\in H^4_p(Q)$. Obviously, $A(\cdot, \cdot)$ is bounded on $H^2_p(Q)$.  By the elliptic estimate \eqref{dh2}, for any $\lambda>0$, $\phi\in H^2_p(Q)$, there is some $\eta'$ such that $A(\phi, \phi)+\eta\|\phi\|^2\geq \eta'\|\phi\|_{H^2}$. Thus, it follows from the Lax--Milgram theorem that the self-adjoint operator $k\varepsilon\Delta + \eta I: H^4_p(Q)\to L^2_p(Q)$ is an isomorphism. Then we see that $k\varepsilon \Delta^2: H^4_p(Q)\to L^2_p(Q) $ is a Fredholm operator of index zero (cf. e.g.,  \cite[Section 8]{Z}). The remaining term $R(\psi)$ consists of $\psi$, $\nabla \psi$, $\Delta \psi$ and their integrals as well as differential operators $\Delta$, $\nabla$. Therefore, $R(\psi)$ is a compact operator from $H^4_p(Q)$ to $L^2_p(Q)$ for $\psi\in H^4_p(Q)$. they rms of $\Delta \phi$, $\nabla \phi$ and $\phi$ that belong to $H^2_p(Q)$ when $\phi\in H^4_p(Q)$.

 As a consequence, for any $\psi\in H^4_p(Q)$, $E''(\psi)$ is a compact perturbation of a Fredholm operator of index zero from  $H^4_p(Q)$ to $L^2_p(Q)$, then itself is also a Fredholm operator of index zero from  $H^4_p(Q)$ to $L^2_p(Q)$ (cf. e.g, \cite[Section 8]{Z}). The proof is complete.
  \end{proof}

Using Proposition \ref{anal}, we can infer from the abstract result
\cite[Corollary 3.11]{Chill}  that the following result holds

\bt
   \label{ls1}
   Suppose $n=3$. Let $\psi$ be a critical point of energy $E$.
   Then, there exist constants $\beta_1>0$, $\theta\in(0, \frac12)$
   depending on $\psi$ such that for any $\phi\in H^4_p(Q)$
   with $\|\phi-\psi\|_{H^4}<\beta_1$, there holds
   \be
   \Big\|\dfrac{\delta
E}{\delta\phi}\Big\| \geq  |E(\phi)-E(\psi)|^{1-\theta}.
   \label{LoSi1}
   \ee
   \et

 \textbf{Proof of Lemma \ref{ls}}.    Based on Theorem \ref{ls1}, we now relax the smallness condition and show that \eqref{LoSi} holds if one only requires that $\phi$ falls into a certain $H^2$-neighbourhood of $\psi$.
  For any $\phi\in H^4_p(Q)$, using the regularity theory for elliptic
   problem, we can see that
   \be \|\phi-\psi\|_{H^4}\leq M(\|\Delta^2 (\phi-
   \psi)\|+ \|\phi-\psi\|),
   \label{55}
   \ee
where $M$ is a constant independent of $\phi$ and $\psi$. If
$\|\phi-\psi\|_{H^2}\leq 1$, we take this assumption just to ensure
that that the fact $ \|\phi\|_{H^2}\leq \|\psi\|_{H^2}+1$ depends
only on $\psi$. Similar to \eqref{F1}--\eqref{F3}, we see that
 \be
  |E(\phi)-E(\psi)|^{1-\theta} \leq
 C_1\|\phi-\psi\|^{1-\theta}_{H^2}.
 \ee
 By H\"older inequality and Sobolev embedding theorem $H^2_p(Q)\hookrightarrow L^\infty(Q)$, we get
 \bea
 |B(\phi)-B(\psi)|&\leq& C\|\nabla(\phi+\psi)\|\|\nabla (\phi-\psi)\|+C(\|\phi\|_{L^\infty}^3+\|\psi\|^3_{L^\infty}+ 1)\|\phi-\psi\|_{L^1}\non\\
 &\leq& C\|\phi-\psi\|  _{H^1},\non
 \eea
 \bea
 \|f(\phi-f(\psi)\|&\leq & C\|\Delta\phi-\Delta\psi\|+ C(\|\phi\|_{L^\infty}^2+\|\psi\|^2_{L^\infty}+ 1)\|\phi-\psi\|\non\\
 &\leq & 
 C\|\phi-\psi\|_{H^2},\non
 \eea
 where $C$ only depend on $\|\psi\|_{H^2}$. Recalling the expression of $H(\phi)$ given in \eqref{v1E}, we obtain
 \bea
  && \|H(\phi)-H(\psi)\| \non\\
  &\leq&
  C\|\phi-\psi\|_{L^\infty}\|\nabla \phi\|^2_{L^4}+C\|\psi\|_{L^\infty}\|\nabla (\phi+\psi)\|_{L^4}\|\nabla (\phi-\psi)\|_{L^4}\non\\
  && +C\|\phi\|_{L^\infty}^2\|\Delta(\phi-\psi)\|+\|\Delta \psi\|\|\phi+\psi\|_{L^\infty}\|\phi-\psi\|_{L^\infty}\non\\
  && +C(\|\phi\|_{L^\infty}, \|\psi\|_{L^\infty})\|\phi-\psi\|+M_1\|\phi-\psi\|_{L^1}+M_2|B(\phi)-B(\psi)|\|f(\phi)\|\non\\
  && +M_2(|B(\psi)|+\beta)\|f(\phi)-f(\psi)\|
  \non\\
  &\leq&  C_2\|\phi-\psi\|_{H^2}.\non
 \eea
 Since the above constants $C_1, C_2$ only depend on $\|\psi\|_{H^2}$, there
 exists a (sufficiently small) constant $\beta$ independent of $\phi$
 which satisfies
 $$ 0<\beta< \min\left\{1, \beta_1, \frac{\beta_1}{2M}, \Big(\frac{\beta_1 k\varepsilon}{4M(C_1+C_2)}\Big)^\frac{1}{1-\theta}\right\} $$
 such that if
 $\|\phi-\psi\|_{H^2}<\beta$, then
 \be  \|H(\phi)-H(\psi)\|+ |E(\phi)-E(\psi)|^{1-\theta}< \frac{\beta_1k\varepsilon}{4M}.
  \label{551}
  \ee
For any $\phi\in H^4_p(Q)$ satisfying  $\|\phi-\psi\|_{H^2}<\beta$,
there are only two possibilities: (i) if $\phi$ also satisfies
$\|\phi-\psi\|_{H^4}<\beta_1$, then \eqref{LoSi1} holds; (ii)
otherwise, if $\|\phi-\psi\|_{H^4}\geq \beta_1$,  note that $\psi$
satisfies \eqref{se}, hence we deduce from (\ref{55}) and
(\ref{551}) that
\begin{eqnarray}
 \Big\| \dfrac{\delta
E}{\delta\phi}\Big\|
& = &\|k\varepsilon \Delta^2 (\phi-\psi)+H(\phi)-H(\psi) \|\non\\
&\geq& k\varepsilon\| \Delta^2 (\phi-\psi)\|-\|H(\phi)-H(\psi)\|\nonumber\\
&\geq & \frac{3\beta_1 k\varepsilon}{4M}\|\phi-\psi\|_{H^4}- k\varepsilon\|\phi-\psi\|\nonumber\\
& > & \frac{\beta_1k\varepsilon}{4M}>
|E(\phi)-E(\psi)|^{1-\theta}.\non
\end{eqnarray}
The proof is complete.

\bigskip
\noindent \textbf{Acknowledgments:} The authors are grateful to the
referees for their helpful comments and suggestions. X. Xu would
like to thank Professor X.-Q. Wang for helpful discussions. Part of
the work was done when X. Xu was visiting School of Mathematical
Sciences at Fudan University, whose hospitality is acknowledged. X. Xu would also warmly thank the Center
for Nonlinear Analysis at Carnegie Mellon University (NSF Grants DMS-0405343 and
DMS-0635983), where part of this research was carried out. H.
Wu was partially supported by NSF of China 11001058, Specialized Research Fund for
the Doctoral Program of Higher Education and the
Fundamental Research Funds for the Central Universities. X. Xu was
partially supported by the NSF grant DMS-0806703.


\end{document}